\documentclass{siamltex}%
\usepackage{amsfonts}
\usepackage{amsmath}
\usepackage{amssymb}
\usepackage{graphicx}
\usepackage{enumerate}
\usepackage{color}
\usepackage{hyperref}%
\newtheorem{example}[theorem]{Example}

\newtheorem{remark}[theorem]{Remark}

\newtheorem{teo}[theorem]{Theorem}
\newtheorem{lem}[theorem]{Lemma}
\begin{document}

\title{Control Sets for Bilinear and Affine Systems}
\author{Fritz Colonius (corresponding author)\\Institut f\"{u}r Mathematik, Universit\"{a}t Augsburg, Augsburg, Germany
\and Alexandre J. Santana and Juliana Setti\\Departamento de Matem\'{a}tica, Universidade Estadual de Maring\'{a}\\Maring\'{a}, Brazil}
\maketitle

\today

\textbf{Abstract. }For homogeneous bilinear control systems, the control sets
are characterized using a Lie algebra rank condition for the induced systems
on projective space. This is based on a classical Diophantine approximation
result. For affine control systems, the control sets around the equilibria for
constant controls are characterized with particular attention to the question
when the control sets are unbounded.


\textbf{MSC\ 2020.} 93B05, 34H05, 11D04

\section{Introduction}

We will study controllability properties of affine control systems of the
form
\begin{equation}
\dot{x}(t)=Ax(t)+\sum_{i=1}^{m}u_{i}(t)(B_{i}x(t)+c_{i})+d, \label{affine}%
\end{equation}
where $A,B_{1},\ldots,B_{m}\in\mathbb{R}^{n\times n}$ and $c_{1},\ldots
,c_{m},d$ are vectors in $\mathbb{R}^{n}$. The controls $u=(u_{1},\ldots
,u_{m})$ have values in a set $\Omega\subset\mathbb{R}^{m}$. The set of
admissible controls is $\mathcal{U}=\{u\in L^{\infty}(\mathbb{R}%
,\mathbb{R}^{m})\left\vert u(t)\in\Omega\text{ for almost all }t\right.  \}$
or the set $\mathcal{U}_{pc}$ of all piecewise constant functions defined on
$\mathbb{R}$ with values in $\Omega$.

Controllability properties of bilinear and affine control systems have been
intensely studied in the last 50 years. The classical monograph by Mohler
\cite{Mohler} contains sufficient conditions for complete controllability and
many applications of bilinear control systems. The monograph Elliott
\cite{Elliott} emphasizes the use of matrix Lie groups and Lie semigroups and
contains a wealth of results on the control of bilinear control systems.

Motivated by the Kalman criterion for controllability of linear systems, an
early goal was show that controllability of bilinear control systems (without
control restrictions) has an algebraic characterization. This hope did not
bear out, in spite of many partial results. The present paper is mainly
concerned with the analysis of control sets, that is, maximal subsets of
complete approximate controllability in $\mathbb{R}^{n}$, cf. Definition
\ref{def:Dset} and Colonius and Kliemann \cite{ColK00} for a general theory.

Concerning the literature on controllability properties of affine and bilinear
systems, many contributions are based on their analysis via the theory of
semigroups in Lie groups, this includes Boothby and Wilson \cite{BW}, Bonnard
\cite{Bonn81}, Jurdjevic and Kupka \cite{JurK81}, Gauthier and Bornard
\cite{GauB82}, Bonnard, Jurdjevic, Kupka, and Sallet \cite{BJKS}, Jurdjevic
and Sallet \cite{JurS84}, San Martin \cite{SanM93}.

The main result of Do Rocio, Santana, and Verdi \cite[Theorem 1.3]{DoRoSV09}
concerns a connected semigroup $S$ with nonvoid interior in an affine group
$G=B\rtimes V$, where $V$ is a finite dimensional vector space and $B$ is a
semisimple Lie group that acts transitively on $V\setminus\{0\}$. If the
linear action of the canonical projection $\pi(S)$ on $B$ is transitive on
$V\setminus\{0\}$, then the affine action of $S$ on $V$ is transitive. This
improves an earlier result in \cite{JurS84}. An application to an affine
control system of the form%
\begin{equation}
\dot{x}=Ax+a+uBx+ub\text{ with }u\in\mathbb{R}, \label{affine2}%
\end{equation}
where $A,B\in\mathfrak{sl}(2,\mathbb{R})$ and $a,b\in\mathbb{R}^{2}$, results
in a sufficient controllability criterion in terms of these parameters.

Answering a question by Sachkov \cite{Sach}, Do Rocio, San Martin, and Santana
\cite{DoRoSMS06} prove that systems of the form (\ref{affine2}) with $a=b=0$
and unrestricted control may not be completely controllable on $\mathbb{R}%
^{n}\setminus\{0\}$ while there is no nontrivial proper closed convex cone in
$\mathbb{R}^{n}$ which is positively invariant. For the relation to the
results in the present paper see Remark \ref{Remark_inv_cone} and also
Proposition \ref{Proposition5.14}.

Our results on control sets will also yield some results on controllability on
$\mathbb{R}^{n}$. We do not restrict our attention to the situation where the
system semigroup has nonvoid interior in the system group. Correspondingly,
our main results are not based on methods for semigroups in Lie groups.

In the first part of this paper we discuss control sets for homogeneous
bilinear systems which are a special case of (\ref{affine}) with $c_{1}%
=\dotsb=c_{m}=d=0$. It is well known that, for this class of systems, one can
separate controllability properties into properties concerning the angular
part on the unit sphere $\mathbb{S}^{n-1}$ and the radial part. In particular,
by Bacciotti and Vivalda \cite[Theorem 1]{BacV13} the induced system on
projective space $\mathbb{P}^{n-1}$ is controllable if and only if the induced
system on $\mathbb{S}^{n-1}$ is controllable.

Theorem \ref{Theorem_cones} shows that every control set $_{\mathbb{S}}D$ with
nonvoid interior on $\mathbb{S}^{n-1}$ induces a control set $D$ on
$\mathbb{R}^{n}\setminus\{0\}$ given by the cone generated by $_{\mathbb{S}}D$
provided that exponential growth and decay can be achieved. Here we use a
classical result on Diophantine approximations which allows us to require only
the accessibility rank condition on $\mathbb{S}^{n-1}$ in the interior of
$_{\mathbb{S}}D$. This result is illustrated by two-dimensional examples. For
systems satisfying the accessibility rank condition on projective space, the
control sets on the unit sphere and on $\mathbb{R}^{n}\setminus\{0\}$ are
characterized in Theorem \ref{Theorem_sphere}\ and Theorem \ref{Theorem_95},
respectively. We remark that under the accessibility rank condition in
$\mathbb{R}^{2}$, a complete description of the control sets and of
controllability is given in Ayala, Cruz, Kliemann, and Laura-Guarachi
\cite{ACKL16}. Corollary \ref{Corollary_approximate} characterizes
controllability on $\mathbb{R}^{n}\setminus\{0\}$ for systems satisfying only
the accessibility rank condition on $\mathbb{P}^{n-1}$ using a recent result
by Cannarsa and Sigalotti \cite[Theorem 1]{CanS21} which shows that here
approximate controllability implies controllability.

In the second part we analyze control sets for general affine systems and
their relation to equilibria. If the systems linearized about equilibria are
controllable, Theorem \ref{Theorem3} shows that any pathwise connected set of
equilibria is contained in a control set. Additional assumptions on spectral
properties of the matrices $A(u)=A+\sum_{i=1}^{m}u_{i}B_{i},~u\in\Omega$,
allow us to get more detailed information. In particular, if $0$ is an
eigenvalue of $A(u^{0})$ for some $u^{0}\in\Omega$, one finds an unbounded
control set, cf. Theorem \ref{Theorem_unbounded}. The main open problem for
control sets of affine systems is, if every control set contains an equilibrium.

The contents of this paper are as follows. Section \ref{Section2} describes
basic properties of nonlinear control systems and control sets as well as some
notation for bilinear and affine control systems. Section \ref{Section3}
discusses homogeneous bilinear control systems using their projection to the
unit sphere. Section \ref{Section4} briefly describes equilibria of affine
systems and Section \ref{Section5} presents results on control sets around
such equilibria.

\section{Preliminaries\label{Section2}}

In this section we introduce some terminology and notations for control-affine
systems and discuss special cases of affine control systems.

\subsection{Control sets}

Control-affine systems on a smooth manifold $M$ have the form%
\begin{align}
\dot{x}(t)  &  =f_{0}(x(t))+\sum_{i=1}^{m}u_{i}(t)f_{i}(x(t)),\label{3.1}\\
u  &  \in\mathcal{U}:=\left\{  u\in L^{\infty}(\mathbb{R},\mathbb{R}%
^{m})\left\vert u(t)\in\Omega\,\text{for\ almost all}\,t\in\mathbb{R}\right.
\right\}  ,\nonumber
\end{align}
where $f_{0},f_{1},\ldots,f_{m}$ are smooth vector fields on $M$ and the
control range $\Omega\subset\mathbb{R}^{m}$ is compact with $0\in
\mathrm{int}\left(  \Omega\right)  $. We assume that for every initial state
$x\in M$ and every control function $u\in\mathcal{U}$ there exists a unique
solution $\varphi(t,x,u),t\in\mathbb{R}$, satisfying $\varphi(0,x,u)=x$ of
(\ref{3.1}) depending continuously on $x$. The system with $u\equiv0$ given by%
\begin{equation}
\dot{x}(t)=f_{0}(x(t)) \label{uncontrolled}%
\end{equation}
is called the uncontrolled system. It generates a continuous flow $\varphi
_{t}$ on $M$. For the general theory of nonlinear control systems we refer to
Sontag \cite{Son98} and Jurdjevic \cite{Jurd97}.

The set of points reachable from $x\in M$ and controllable to $x\in M$ up to
time $T>0$ are defined by
\begin{align*}
{\mathcal{O}}_{\leq T}^{+}(x)  &  :=\{y\in M\left\vert \text{there are}\;0\leq
t\leq T\;\text{and}\;u\in\mathcal{U}\;\text{with}\;y=\varphi(t,x,u)\right.
\},\\
{\mathcal{O}}_{\leq T}^{-}(x)  &  :=\{y\in M\left\vert \text{there are}\;0\leq
t\leq T\;\text{and}\;u\in\mathcal{U}\;\text{with}\;x=\varphi(t,y,u)\right.
\},
\end{align*}
resp. Furthermore, the reachable set (or \textquotedblleft positive
orbit\textquotedblright) from $x$ and the set controllable to $x$ (or
\textquotedblleft negative orbit\textquotedblright\ of $x$) are%
\[
\mathcal{O}^{+}(x)=\bigcup\nolimits_{T>0}O_{\leq T}^{+}(x),\quad
\mathcal{O}^{-}(x)=\bigcup\nolimits_{T>0}O_{\leq T}^{-}(x),
\]
resp. The system is called locally accessible in $x$, if $\mathcal{O}_{\leq
T}^{+}(x)$ and $\mathcal{O}_{\leq T}^{-}(x)$ have nonvoid interior for all
$T>0$ and the system is called locally accessible if this holds in every point
$x\in M$. This is guaranteed by the following accessibility rank condition%
\begin{equation}
\dim\mathcal{LA}\left\{  f_{0},f_{1},\ldots,f_{m}\right\}  (x)=\dim M\text{
for all }x\in M; \label{ARC}%
\end{equation}
here $\mathcal{LA}\left\{  f_{0},f_{1},\ldots,f_{m}\right\}  (x)$ is the
subspace of the tangent space $T_{x}M$ corresponding to the vector fields,
evaluated in $x$, in the Lie algebra generated by $f_{0},f_{1},\ldots,f_{m}$.

The trajectories for the convex hull of $\Omega$ can be uniformly approximated
on bounded intervals by the trajectories for $\Omega$. Furthermore,
trajectories for controls in $\mathcal{U}$ can be uniformly approximated on
bounded intervals by trajectories for piecewise constant controls in
$\mathcal{U}_{pc}$.

The following definition introduces subsets of complete approximate
controllability which are of primary interest in the present paper.

\begin{definition}
\label{def:Dset}A nonvoid set $D\subset M$ is called a control set of system
(\ref{3.1}) if it has the following properties: (i) for all $x\in D$ there is
a control function $u\in\mathcal{U}$ such that $\varphi(t,x,u)\in D$ for all
$t\geq0$, (ii) for all $x\in D$ one has $D\subset\overline{\mathcal{O}^{+}%
(x)}$, and (iii) $D$ is maximal with these properties, that is, if $D^{\prime
}\supset D$ satisfies conditions (i) and (ii), then $D^{\prime}=D$.

A control set $D\subset M$ is called an invariant control set if $\overline
{D}=\overline{\mathcal{O}^{+}(x)}$ for all $x\in D$. All other control sets
are called variant.
\end{definition}

If the intersection of two control sets is nonvoid, the maximality property
(ii) implies that they coincide. If the system is locally accessible in all
$x\in\mathrm{int}\left(  D\right)  $, then $\mathrm{int}\left(  D\right)
\subset\mathcal{O}^{+}(x)$ for all $x\in D$ and $D=\mathcal{O}^{-}%
(x)\cap\overline{\mathcal{O}^{+}(x)}$ for every $x\in\mathrm{int}\left(
D\right)  $. The control sets with nonvoid interior for piecewise constant
controls in $\mathcal{U}_{pc}$ coincide with those for controls in
$\mathcal{U}$. For these and further properties of control sets, we refer to
Colonius and Kliemann \cite[Chapters 3 and 4]{ColK00}.

The following lemma shows that the controllable set of system (\ref{affine})
coincides with the reachable set of the time reversed system.

\begin{lem}
\label{Lemma_time_reversal}Consider together with system (\ref{3.1}) the time
reversed system
\begin{equation}
\dot{x}(t)=-f_{0}(x(t))-\sum_{i=1}^{m}v_{i}(t)f_{i}(x(t)),\quad v\in
\mathcal{U}. \label{reversed}%
\end{equation}
We denote by $\mathcal{O}_{1}^{+}(x)$ and $\mathcal{O}_{1}^{-}(x)$ the
reachable set from $x$ and the controllable set to $x,$ determined by the
system (\ref{3.1}), respectively, and by $\mathcal{O}_{2}^{+}(x)$ and
$\mathcal{O}_{2}^{-}(x)$ the reachable set from $x$ and the controllable set
to $x,$ determined by the system (\ref{reversed}), respectively. Then
$\mathcal{O}_{1}^{+}(x)=\mathcal{O}_{2}^{-}(x)$ and $\mathcal{O}_{1}%
^{-}(x)=\mathcal{O}_{2}^{+}(x)$.
\end{lem}

\begin{proof}
For $y=\varphi(T,x,u)\in\mathcal{O}_{1}^{+}(x)$, the absolutely continuous
function $\psi(t):=\allowbreak\varphi(T-t,x,u(T-\cdot)),t\in\lbrack0,T]$,
satisfies $\psi(0)=y,\psi(T)=x$. It is a solution of (\ref{reversed}) with
$v(t):=u(T-t),t\in\lbrack0,T]$, since for almost all $t\in\lbrack0,T]$%
\begin{align*}
&  \dot{\psi}(t)=\frac{d}{dt}\varphi(T-t,y,u(T-\cdot))\\
&  =-f_{0}(\varphi(T-t,y,u(T-\cdot)))-\sum_{i=1}^{m}u_{i}(T-t)f_{i}%
(\varphi(T-t,y,u(T-\cdot))\\
&  =-f_{0}(\psi(t))-\sum_{i=1}^{m}v_{i}(t)f_{i}(\psi(t)).
\end{align*}
Thus $\mathcal{O}_{1}^{+}(x)\subset\mathcal{O}_{2}^{-}(x)$. The other
inclusions follow analogously.
\end{proof}

\subsection{Affine and bilinear control systems}

Frequently, we abbreviate%
\begin{equation}
A(u):=A+\sum_{i=1}^{m}u_{i}B_{i}\text{ for }u\in\Omega\text{ and }%
C:=(c_{1},\ldots,c_{m})\in\mathbb{R}^{n\times m}, \label{A(u)}%
\end{equation}
hence the columns of $C$ are given by the $c_{i}$. Then (\ref{affine}) can be
written as%
\[
\dot{x}(t)=A(u(t))x(t)+Cu(t)+d.
\]
A special case are bilinear control systems obtained for $d=0$, i.e.
\begin{equation}
\dot{x}(t)=Ax(t)+\sum_{i=1}^{m}u_{i}(t)(B_{i}x(t)+c_{i})=A(u(t))x(t)+Cu(t),
\label{bilinear}%
\end{equation}
and homogeneous bilinear systems of the form%
\begin{equation}
\dot{x}(t)=Ax(t)+\sum_{i=1}^{m}u_{i}(t)B_{i}x(t)=A(u(t))x(t).
\label{bilinear_h}%
\end{equation}
For fixed control $u\in\mathcal{U}$ (\ref{affine}) is a nonautonomous
inhomogeneous linear differential equation. Denote by $\Phi_{u}(t,s)\in
\mathbb{R}^{n\times n}$ the principal matrix solution, i.e., the solution of
\[
\frac{d}{dt}\Phi_{u}(t,s)=A(u(t))\Phi_{u}(t,s),\quad\Phi_{u}(s,s)=I.
\]
The solutions $\varphi(t,x_{0},u),t\in\mathbb{R}$, of (\ref{affine}) with
initial condition $\varphi(0,x_{0},u)=x_{0}\in\mathbb{R}^{n}$ are given by%
\[
\varphi(t,x_{0},u)=\Phi_{u}(t,0)x_{0}+\int_{0}^{t}\Phi_{u}%
(t,s)[Cu(s)+d]ds,\quad t\in\mathbb{R},
\]
and, in particular, the solutions of (\ref{bilinear_h}) are%
\[
\varphi(t,x_{0},u)=\Phi_{u}(t,0)x_{0},\quad t\in\mathbb{R}.
\]
This readily implies for $\alpha\in\mathbb{R}$%
\begin{equation}
\varphi(t,\alpha x_{0},u)=\Phi_{u}(t,0)\alpha x_{0}=\alpha\varphi(t,x_{0},u).
\label{homogene}%
\end{equation}

\section{Control sets for homogeneous bilinear systems\label{Section3}}

We consider homogeneous bilinear control systems of the form (\ref{bilinear_h}%
) and describe their control sets.

Since for fixed control $u$, the corresponding differential equations are
homogeneous, their controllability properties can often be split into
controllability properties for the angles and the radii separately; cf., e.g.,
Colonius and Kliemann \cite[Chapter 7]{ColK00}. Denote the projection of
$\mathbb{R}^{n}$ to the Euclidean unit sphere $\mathbb{S}^{n-1}$ by $\pi$ and
the projection to real projective space $\mathbb{P}^{n-1}$ (obtained by
identifying opposite points on the sphere) by $\mathbb{P}$. For a trajectory
of (\ref{bilinear_h}) define
\[
s(t):=\pi(x(t))=\frac{x(t)}{\left\Vert x(t)\right\Vert },\quad t\in\mathbb{R}.
\]
The projected trajectories are trajectories of control-affine systems on
$\mathbb{S}^{n-1}$ given by%
\begin{align}
\dot{s}(t) &  =h(u(t),s(t))=h_{0}(s(t))+\sum_{i=1}^{m}u_{i}(t)h_{i}%
(s(t)),\label{sphere}\\
h_{0}(s) &  =As-s^{\top}As\cdot s,\quad h_{i}(s)=B_{i}s-s^{\top}B_{i}s\cdot
s\text{ for }i=1,\ldots,m.\nonumber
\end{align}
The vector fields of the system on $\mathbb{S}^{n-1}$ are obtained by
subtracting the radial component. The solutions will be denoted by
$s(t,s_{0},u),t\in\mathbb{R}$. One also obtains an induced control system on
projective space $\mathbb{P}^{n-1}$ with vector fields $\mathbb{P}h(u,\cdot)$
since $h_{i}(s)=-h_{i}(-s)$ for all $i$.

Since bilinear control systems as well as their projections to $\mathbb{S}%
^{n-1}$ and $\mathbb{P}^{n-1}$ are analytic, for these systems, local
accessibility is equivalent to the corresponding accessibility rank condition
(\ref{ARC}); cf. Sontag \cite[Theorem 12 on p. 179]{Son98}.

We note the following simple result showing a first relation between control
sets on $\mathbb{R}^{n}$ and control sets on $\mathbb{S}^{n-1}$.

\begin{proposition}
Suppose that $D\subset\mathbb{R}^{n}\setminus\{0\}$ is a control set of system
(\ref{bilinear_h}). Then the projection $\mathbb{P}(D)$ to projective space
$\mathbb{P}^{n-1}$ is contained in a control set $_{\mathbb{P}}D$ for the
induced system on $\mathbb{P}^{n-1}$, and the projection $\pi(D)$ to the unit
sphere $\mathbb{S}^{n-1}$ is contained in a control set $_{\mathbb{S}}D$ for
the induced system (\ref{sphere}) on $\mathbb{S}^{n-1}$. If $D$ has nonvoid
interior, then also $_{\mathbb{P}}D$ and $\,_{\mathbb{S}}D$ have nonvoid interiors.
\end{proposition}

\begin{proof}
The assertions immediately follow from the definitions and the fact that the
projections $\pi$ and $\mathbb{P}$ are open.
\end{proof}

Next we will analyze when a control set on the unit sphere $\mathbb{S}^{n-1}$
generates a control set on $\mathbb{R}^{n}$. This result is based on a
Diophantine approximation result used for Lemma \ref{Lemma_dio}.

\begin{theorem}
\label{Theorem_cones}Let $_{\mathbb{S}}D$ be a control set with nonvoid
interior for the system on the unit sphere $\mathbb{S}^{n-1}$ and suppose that

(i) every point in $\mathrm{int}\left(  _{\mathbb{S}}D\right)  $ is locally acessible;

(ii) there are $\alpha_{0}^{+}>1,~\delta_{0}>0$, and $\alpha^{-}\in(0,1)$ such
that for all $\alpha^{+}\in(\alpha_{0}^{+},\alpha_{0}^{+}+\delta_{0})$ there
are points $s^{+},s^{-}\in\mathrm{int}\left(  _{\mathbb{S}}D\right)  $,
controls $u^{+},u^{-}\in\mathcal{U}$, and times $\sigma^{+},\sigma^{-}>0$ with%
\begin{equation}
\varphi(\sigma^{+},s^{+},u^{+})=\alpha^{+}s^{+},\quad\varphi(\sigma^{-}%
,s^{-},u^{-})=\alpha^{-}s^{-}. \label{R1}%
\end{equation}
Then the cone $\{\alpha s\in\mathbb{R}^{n}\left\vert \alpha>0,s\in
\,_{\mathbb{S}}D\right.  \}$ is a control set in $\mathbb{R}^{n}$ with nonvoid interior.
\end{theorem}

\begin{remark}
The proof of Theorem \ref{Theorem_cones} will show that we can replace
assumption (ii) by the following assumption:

(ii)' there are $\alpha^{+}>1,\delta_{0}\in(0,1)$, and $\alpha_{0}^{-}%
\in(0,1-\delta_{0})$ such that for all $\alpha^{-}\in(\alpha_{0}^{-}%
,\alpha_{0}^{-}+\delta_{0})$ there are points $s^{+},s^{-}\in\mathrm{int}%
\left(  _{\mathbb{S}}D\right)  $, controls $u^{+},u^{-}\in\mathcal{U}$, and
times $\sigma^{+},\sigma^{-}>0$ with (\ref{R1}).
\end{remark}

\begin{proof}
First observe that (\ref{R1}) implies for the projected system on
$\mathbb{S}^{n-1}$%
\[
s(\sigma^{+},s^{+},u^{+})=s^{+},\quad s(\sigma^{-},s^{-},u^{-})=s^{-}.
\]
Hence we get periodic solutions in $\mathrm{int}\left(  _{\mathbb{S}}D\right)
\subset\mathbb{S}^{n-1.}$.

\textbf{Step 1:} Let $s_{0}\in\mathrm{int}\left(  _{\mathbb{S}}D\right)  $.
Then for every $x_{0}\in l:=\{\alpha s_{0}\in\mathbb{R}^{n}\left\vert
\alpha>0\right.  \}$ the closure of the reachable set from $x_{0}$ contains
the half-line $l$.

For the proof of this claim, consider arbitrary points $x_{0}=\alpha_{0}%
s_{0},x_{1}=\alpha_{1}s_{0}\in l$ with $\alpha_{0},\alpha_{1}>0$. The strategy
is to steer the system from $s_{0}$ to $s^{+}$, then to go $k$ times through
the periodic trajectory for $u^{+}$, then to steer the system to $s^{-}$, go
$\ell$ times through the periodic trajectory for $u^{-}$, and finally steer
the system back to $s_{0}$. The numbers $k,\ell\in\mathbb{N}$ will be adjusted
such that the corresponding trajectories in $\mathbb{R}^{n}$ starting in
$x_{0}$ approach $x_{1}$.

By local accessibility in $\mathrm{int}\left(  _{\mathbb{S}}D\right)  $ there
are times $\tau_{1},\tau_{2},\tau_{3}>0$ and controls $v^{1},v^{2},\allowbreak
v^{3}\in\mathcal{U}$ with%
\[
s(\tau_{1},s_{0},v^{1})=s^{+},\quad s(\tau_{2},s^{+},v^{2})=s^{-},\quad
s(\tau_{3},s^{-},v^{3})=s_{0}.
\]
One finds for the system in $\mathbb{R}^{n}$ numbers $\beta_{1},\beta
_{2},\beta_{3}>0$ with%
\[
\varphi(\tau_{1},x_{0},v^{1})=\varphi(\tau_{1},\alpha_{0}s_{0},v^{1}%
)=\beta_{1}s^{+},~\varphi(\tau_{2},s^{+},v^{2})=\beta_{2}s^{-},~\varphi
(\tau_{3},s^{-},v^{3})=\beta_{3}s_{0}.
\]
Now define for $k,\ell\in\mathbb{N}$ a control function $w^{k,\ell}$ by%
\begin{align*}
w^{k,\ell}(t)  &  =v^{1}(t)\text{ for }t\in\lbrack0,\tau_{1}],\\
w^{k,\ell}(t)  &  =u^{+}(t-(\tau_{1}+(i-1)\sigma^{+}))\text{ for }t\in
(\tau_{1}+(i-1)\sigma^{+},\tau_{1}+i\sigma^{+}],~i=1,\ldots,k,\\
w^{k,\ell}(t)  &  =v^{2}(t-(\tau_{1}+k\sigma^{+}))\text{ for }t\in(\tau
_{1}+k\sigma^{+},\tau_{1}+k\sigma^{+}+\tau_{2}],\\
w^{k,\ell}(t)  &  =u^{-}(t-(\tau_{1}+k\sigma^{+}+\tau_{2}+(i-1)\sigma^{-}))\\
\text{for }t  &  \in(\tau_{1}+k\sigma^{+}+\tau_{2}+(i-1)\sigma^{-},\tau
_{1}+k\sigma^{+}+\tau_{2}+i\sigma^{-}],\quad i=1,\ldots,\ell,\\
w^{k,\ell}(t)  &  =v^{3}(t-(\tau_{1}+k\sigma^{+}+\tau_{2}+\ell\sigma^{-})),\\
\text{for }t  &  \in(\tau_{1}+k\sigma^{+}+\tau_{2}+\ell\sigma^{-},\tau
_{1}+k\sigma^{+}+\tau_{2}+\ell\sigma^{-}+\tau_{3}].
\end{align*}
The corresponding trajectory on $\mathbb{S}^{n-1}$ is periodic and satisfies%
\begin{align*}
s(\tau_{1}+i\sigma^{+},s_{0},w^{k,\ell})  &  =s^{+}\text{ for }i=0,1,\ldots
,k,\\
s(\tau_{1}+k\sigma^{+}+\tau_{2}+i\sigma^{-},s_{0},w^{k,\ell})  &  =s^{-}\text{
for }i=0,1,\ldots,\ell,\\
s(\tau_{1}+k\sigma^{+}+\tau_{2}+\ell\sigma^{-}+\tau_{3},s_{0},w^{k,\ell})  &
=s_{0},
\end{align*}
and for the corresponding trajectory on $\mathbb{R}^{n}$ one finds using
(\ref{homogene})%
\begin{align*}
\varphi(\tau_{1}+i\sigma^{+},x_{0},w^{k,\ell})  &  =\left(  \alpha^{+}\right)
^{i}\beta_{1}s^{+}\text{ for }i=0,1,\ldots,k,\\
\varphi(\tau_{1}+k\sigma^{+}+\tau_{2}+i\sigma^{-},x_{0},w^{k,\ell})  &
=\left(  \alpha^{-}\right)  ^{i}\beta_{2}\left(  \alpha^{+}\right)  ^{k}%
\beta_{1}s^{-}\text{ for }i=0,1,\ldots,\ell,\\
\varphi(\tau_{1}+k\sigma^{+}+\tau_{2}+\ell\sigma^{-}+\tau_{3},x_{0},w^{k,\ell
})  &  =\beta_{3}\left(  \alpha^{-}\right)  ^{\ell}\beta_{2}\left(  \alpha
^{+}\right)  ^{k}\beta_{1}s_{0}.
\end{align*}
Recall that our goal is to reach $x_{1}=\alpha_{1}s_{0}$ approximately. We
apply Lemma\ \ref{Lemma_dio} with $a=\alpha^{+},~b=\left(  \alpha^{-}\right)
^{-1}$, and $c=\alpha_{1}\left(  \beta_{3}\beta_{2}\beta_{1}\right)  ^{-1}$,
where we choose $\alpha^{+}\in(\alpha_{0}^{+},\alpha_{0}^{+}+\delta_{0})$ such
that $\frac{\log b}{\log a}=\frac{-\log\alpha^{-}}{\log\alpha^{+}}$ is
irrational. Thus for every $\varepsilon>0$ there are $k,\ell\in\mathbb{N}$
with%
\[
\left\vert \left(  \alpha^{+}\right)  ^{k}\left(  \alpha^{-}\right)  ^{\ell
}-\alpha_{1}\left(  \beta_{3}\beta_{2}\beta_{1}\right)  ^{-1}\right\vert
=\left\vert \frac{\left(  \alpha^{+}\right)  ^{k}}{\left(  \alpha^{-}\right)
^{-\ell}}-\alpha_{1}\left(  \beta_{3}\beta_{2}\beta_{1}\right)  ^{-1}%
\right\vert <\varepsilon,
\]
hence for all $\varepsilon>0$ there are $k,\ell\in\mathbb{N}$ with%
\[
\left\vert \beta_{3}\beta_{2}\beta_{1}\left(  \alpha^{+}\right)  ^{k}\left(
\alpha^{-}\right)  ^{\ell}-\alpha_{1}\right\vert <\varepsilon.
\]
It follows that for some $\delta\in(-\varepsilon,\varepsilon)$ one can choose
$k,\ell$ such that%
\[
\varphi(\tau_{1}+k\sigma^{+}+\tau_{2}+\ell\sigma^{-}+\tau_{3},x_{0},w^{k,\ell
})=\beta_{3}\beta_{2}\beta_{1}\left(  \alpha^{-}\right)  ^{\ell}\left(
\alpha^{+}\right)  ^{k}s_{0}=(\alpha_{1}+\delta)s_{0}.
\]
Since $\varepsilon>0$ is arbitrary, it follows that $x_{1}=\alpha_{1}s_{0}$ is
in the closure of the reachable set of $x_{0}$ and hence $l$ is contained the
closure of the reachable set from $x_{0}$.

\textbf{Step 2:} Let $x_{1},x_{2}\in\{\alpha s\in\mathbb{R}^{n}\left\vert
\alpha>0,s\in\,_{\mathbb{S}}D\right.  \}$, hence there are $\alpha_{1}%
,\alpha_{2}>0$ and $s_{1},s_{2}\in\,_{\mathbb{S}}D$ with $x_{1}=\alpha
_{1}s_{1}$ and $x_{2}=\alpha_{2}s_{2}$. Then there are a control $u_{1}$ and a
time $t_{1}\geq0$ with $s(t_{1},s_{1},u_{1})=s_{0}$, hence $\varphi
(t_{1},x_{1},u_{1})=\gamma_{1}s_{0}\in l$ for some $\gamma_{1}>0$. Since
$s_{0},s_{2}\in\mathrm{\,}_{\mathbb{S}}D$ one finds, for $\varepsilon>0$, a
control $u_{2}$ and a time $t_{2}\geq0$ such that, for $s_{3}:=s(t_{2}%
,s_{0},u_{2})$,%
\[
\left\Vert s_{3}-s_{2}\right\Vert <\varepsilon/\alpha_{2}\text{ and
}\left\Vert \alpha_{2}s_{3}-x_{2}\right\Vert =\left\Vert \alpha_{2}%
s_{3}-\alpha_{2}s_{2}\right\Vert <\varepsilon.
\]
The trajectory in $\mathbb{R}^{n}$ satisfies $\varphi(t_{2},s_{0}%
,u_{2})=\gamma_{2}s_{3}$ for some $\gamma_{2}>0$. By (\ref{homogene}) it
follows that%
\[
\varphi(t_{2},\frac{\alpha_{2}}{\gamma_{2}}s_{0},u_{2})=\frac{\alpha_{2}%
}{\gamma_{2}}\gamma_{2}s_{3}=\alpha_{2}s_{3}.
\]
Step 1 implies that one finds arbitrarily close to $\frac{\alpha_{2}}%
{\gamma_{2}}s_{0}\in l$ points in the reachable set from $\gamma_{1}s_{0}$,
hence in the reachable set from $x_{1}$. By continuous dependence on the
initial value, it follows that under the control $u_{2}$ points in the
reachable set from $x_{1}$ are steered into the $\varepsilon$-neighborhood of
$x_{2}$. Since $\varepsilon>0$ is arbitrary, this shows that $x_{2}$ is in the
closure of the reachable set from $x_{1}.$

\textbf{Step 3}: We have shown that the cone $D^{\prime}:=\{\alpha
s\in\mathbb{R}^{n}\left\vert \alpha>0,s\in\,_{\mathbb{S}}D\right.  \}$ is a
set of complete approximate controllability. It is maximal with this property,
since any set of approximate controllability in $\mathbb{R}^{n}$ projects to a
set of approximate controllability in $\mathbb{S}^{n-1}$, and $_{\mathbb{S}}D$
is a maximal set of approximate controllability. Finally, for every point
$x\in D^{\prime}$ there is a control $u$ with $\varphi(t,x,u)\in D^{\prime}$
for all $t\geq0$, since this holds in $_{\mathbb{S}}D$. Hence the cone
$D^{\prime}$ is a control set and it has a nonvoid interior.
\end{proof}

Step 1 in the proof above is based on the following lemma which uses a
Diophantine approximation property.

\begin{lemma}
\label{Lemma_dio}Let $a,b,c$ be real numbers with $a,b>1,c>0$, and $\frac{\log
b}{\log a}\in\mathbb{R}\setminus\mathbb{Q}$. Then for every $\varepsilon>0$
there are $k,\ell\in\mathbb{N}$ such that $\left\vert a^{k}b^{-\ell
}-c\right\vert <\varepsilon$.
\end{lemma}

\begin{proof}
Since the logarithm is continuously invertible, it suffices to show that for
every $\varepsilon>0$ there are $k,\ell\in\mathbb{N}$ with
\[
\varepsilon>\left\vert \log(a^{k}b^{-\ell})-\log c\right\vert =\left\vert
k\log a-\ell\log b-\log c\right\vert ,
\]
or, dividing by $\log a>0$,%
\[
\left\vert k-\ell\frac{\log b}{\log a}-\frac{\log c}{\log a}\right\vert
<\frac{\varepsilon}{\log a}.
\]
We use the following Diophantine approximation result which is due to
Tchebychef \cite[Th\'{e}or\`{e}me, p. 679]{Tcheb}: For any irrational number
$\alpha$ and any $\beta\in\mathbb{R}$ the inequality $x\left\vert y-\alpha
x-\beta\right\vert <2$ has an infinite number of solutions in $x\in
\mathbb{N},y\in\mathbb{Z}$. Observe that here also $y\in\mathbb{N}$ if
$\alpha>0$, since then $\mathrm{sgn}(y)=\mathrm{sgn}(\alpha x)=\mathrm{sgn}%
(x)=1$. For an application to the problem above, let $\alpha=\frac{\log
b}{\log a}>0,~\beta=\frac{\log c}{\log a},x=\ell,y=k$. One obtains that
\[
\ell\left\vert k-\ell\frac{\log b}{\log a}-\frac{\log c}{\log a}\right\vert
<2
\]
has an infinite number of solutions $k,\ell\in\mathbb{N}$. Choosing $\ell$
large enough such that $\frac{2\log a}{\ell}<\varepsilon$ and dividing by
$\ell$ one gets, as desired,%
\[
\left\vert k-\ell\frac{\log b}{\log a}-\frac{\log c}{\log a}\right\vert
<\frac{2}{\ell}<\frac{\varepsilon}{\log a}.
\]

\end{proof}

\begin{remark}
The Diophantine approximation result used above is closely related to a
theorem due to Minkowski on inhomogeneous linear Diophantine approximation,
cf. Cassels, \cite[Theorem I in Chapter III]{Cassels}. Here the existence of
integers $x,y$ solving $x\left\vert y-\alpha x-\beta\right\vert <\frac{1}{4}$
is established, but not the existence of infinitely many pairs $x,y$ with this
property, as required for the proof above.
\end{remark}

\begin{remark}
\label{Remark3.4}Suppose that for a control set $_{\mathbb{S}}D$ on the unit
sphere, every point in the interior is locally acessible and there are control
values $u^{\pm}\in\mathrm{int}\left(  \Omega\right)  $ such that $A(u^{+})$
has an eigenvalue $\lambda^{+}>0$ and $A(u^{-})$ has an eigenvalue
$\lambda^{-}<0$ with eigenspaces satisfying $E(\lambda^{\pm})\cap
\mathrm{int}\left(  _{\mathbb{S}}D\right)  \not =\varnothing$. Then assumption
(ii) of Theorem \ref{Theorem_95} holds. In fact, all points $s^{\pm}\in
E(\lambda^{\pm})\cap\mathrm{int}\left(  _{\mathbb{S}}D\right)  \,$ are
equilibria for the induced system on $\mathbb{S}^{n-1}$ with $A(u^{\pm}%
)s^{\pm}=\lambda^{\pm}s^{\pm}$. This implies for all $\sigma^{\pm}>0$ and the
constant controls $u^{\pm}\in\Omega$ that%
\begin{align*}
\varphi(\sigma^{+},s^{+},u^{+})  &  =\alpha_{0}^{+}s^{+}\text{ with }%
\alpha_{0}^{+}:=e^{\lambda^{+}\sigma^{+}}>1,\\
\varphi(\sigma^{-},s^{-},u^{-})  &  =\alpha^{-}s^{-}\text{ with }\alpha
^{-}:=e^{\lambda^{-}\sigma^{-}}<1.
\end{align*}
This follows, since the solutions of $\dot{x}=A(u^{\pm})x,~x(0)=s^{\pm}$, are
given by%
\[
\varphi(t,s^{\pm},u^{\pm})=e^{A(u^{\pm})t}s^{\pm}=e^{\lambda^{\pm}t}s^{\pm}.
\]
Varying $\sigma^{+}$, we get that $\varphi(\sigma^{+},s^{+},u^{+})=\alpha
^{+}s^{+}$ for all $\alpha^{+}\in(\alpha_{0}^{+},\alpha_{0}^{+}+\delta_{0})$
and some $\delta_{0}>0$.
\end{remark}

The following two examples illustrate Theorem \ref{Theorem_cones}. We consider
problems in $\mathbb{R}^{2}$ where the induced system on the unit circle is
not locally accessible. First let $A$ be given in Jordan normal form
$A=\left[
\begin{array}
[c]{cc}%
\lambda_{1} & 0\\
0 & \lambda_{2}%
\end{array}
\right]  $ and let the matrices $B_{1}$ and $B_{2}$ be diagonal. The situation
is a bit more complicated than in Remark \ref{Remark3.4}, since the
intersections of the relevant eigenspaces with the unit sphere yield boundary
points of the control set $_{\mathbb{S}}D$.

\begin{example}
\label{Ex1}Consider a system of the form%
\begin{equation}
\left[
\begin{array}
[c]{c}%
\dot{x}\\
\dot{y}%
\end{array}
\right]  =\left(  \left[
\begin{array}
[c]{cc}%
\lambda_{1} & 0\\
0 & \lambda_{2}%
\end{array}
\right]  +u(t)\left[
\begin{array}
[c]{cc}%
b_{11} & 0\\
0 & b_{21}%
\end{array}
\right]  +v(t)\left[
\begin{array}
[c]{cc}%
b_{12} & 0\\
0 & b_{22}%
\end{array}
\right]  \right)  \left[
\begin{array}
[c]{c}%
x\\
y
\end{array}
\right]  , \label{Exa1}%
\end{equation}
with $\lambda_{1},\lambda_{2}\in\mathbb{R}$ and control values $(u(t),v(t))\in
\Omega\subset\mathbb{R}^{2}$. This can be written as%
\[
\left[
\begin{array}
[c]{c}%
\dot{x}\\
\dot{y}%
\end{array}
\right]  =\left[
\begin{array}
[c]{cc}%
\lambda_{1}+b_{11}u+b_{12}v & 0\\
0 & \lambda_{2}+b_{21}u+b_{22}v
\end{array}
\right]  \left[
\begin{array}
[c]{c}%
x\\
y
\end{array}
\right]  =A(u,v)\left[
\begin{array}
[c]{c}%
x\\
y
\end{array}
\right]  .
\]
For all $(u,v)\in\Omega$ the eigenvalues $\mu_{1}(u,v)=\lambda_{1}%
+b_{11}u+b_{12}v$ and $\mu_{2}(u,v)=\lambda_{2}+b_{21}u+b_{22}v$ of $A(u,v)$
have the eigenspaces $\mathbb{R}\times\{0\}$ and $\{0\}\times\mathbb{R}$,
resp. Assume that there are control values $(u_{1},v_{1}),(u_{2},v_{2}%
)\in\Omega$ with%
\begin{equation}
\mu_{1}(u_{1},v_{1})>0,\ \mu_{2}(u_{1},v_{1})<0\text{ and }\mu_{1}(u_{2}%
,v_{2})<0,\ \mu_{2}(u_{2},v_{2})>0. \label{Ex1_A1}%
\end{equation}
For $(u_{1},v_{1})$ the eigenspace $\mathbb{R}\times\{0\}$ is attracting and
for $(u_{2},v_{2})$ the eigenspace $\{0\}\times\mathbb{R}$ is attracting. One
easily verifies that on the unit circle $\mathbb{S}^{1}$ there are four open
and invariant control sets $_{\mathbb{S}}D_{i},i=1,\ldots,4$, with nonvoid
interior on the unit sphere separated by the four points in the intersection
of the eigenspaces $\mathbb{R}\times\{0\}$ and $\{0\}\times\mathbb{R}$ with
$\mathbb{S}^{1}$. The four points in this intersection are invariant for all
$(u,v)$, hence they are not locally accessible, while every point in the
control sets is locally accessible.

In order to verify condition (\ref{R1}), assume that there is $(u_{3}%
,v_{3})\in\Omega$ with%
\begin{equation}
\mu_{1}(u_{3},v_{3})=0\text{ and }\mu_{2}(u_{3},v_{3})>0. \label{Ex1_A2}%
\end{equation}
Let $\tau_{1}>0$, and define $\tau_{2}:=\tau_{1}\frac{\mu_{1}(u_{1},v_{1}%
)-\mu_{2}(u_{1},v_{1})}{\mu_{2}(u_{3},v_{3})}>0$ and%
\[
(u^{+}(t),v^{+}(t)):=\left\{
\begin{array}
[c]{lll}%
(u_{1},v_{1}) & \text{for} & t\in\lbrack0,\tau_{1}]\\
(u_{3},v_{3}) & \text{for} & t\in(\tau_{1},\tau_{2}+\tau_{1}]
\end{array}
\right.  .
\]
Fix a point $s^{+}\in\,_{\mathbb{S}}D_{i}$. Then it follows that
\begin{align*}
\varphi(\tau_{2}+\tau_{1},s^{+},u^{+},v^{+})  &  =\varphi(\tau_{2}%
,\varphi(\tau_{1},s^{+},u_{1},v_{1}),u_{3},v_{3})\\
&  =\left[
\begin{array}
[c]{cc}%
e^{0} & 0\\
0 & e^{\tau_{2}\mu_{2}(u_{3},v_{3})}%
\end{array}
\right]  \left[
\begin{array}
[c]{c}%
e^{\tau_{1}\mu_{1}(u_{1},v_{1})}\\
e^{\tau_{1}\mu_{2}(u_{1},v_{1})}%
\end{array}
\right]  s^{+}=e^{\tau_{1}\mu_{1}(u_{1},v_{1})}s^{+}.
\end{align*}
Since $\tau_{1}>0$ is arbitrary, the first equality in (\ref{R1}) holds with
$\sigma^{+}=\tau_{2}+\tau_{1}$ and $\alpha^{+}=e^{\tau_{1}\mu_{1}(u_{1}%
,v_{1})}>1$.

Analogously, fix a point $s^{-}\in\,_{\mathbb{S}}D_{i}$. Assume that there is
$(u_{4},v_{4})\in\Omega$ with%
\begin{equation}
\mu_{1}(u_{4},v_{4})=0,~\mu_{2}(u_{4},v_{4})<0. \label{Ex1_A3}%
\end{equation}
Define, with $\tau_{1}>0$ and $\tau_{3}:=\tau_{1}\frac{\mu_{1}(u_{2}%
,v_{2})-\mu_{2}(u_{2},v_{2})}{\mu_{2}(u_{4},v_{4})}>0$,%
\[
(u^{-}(t),v^{-}(t))=\left\{
\begin{array}
[c]{lll}%
(u_{2},v_{2}) & \text{for} & t\in\lbrack0,\tau_{1}]\\
(u_{4},v_{4}) & \text{for} & t\in(\tau_{1},\tau_{3}+\tau_{1}]
\end{array}
\right.  .
\]
Then it follows that%
\begin{align*}
\varphi(\tau_{3}+\tau_{1},s^{-},u^{-},v^{-})  &  =\varphi(\tau_{3}%
,\varphi(\tau_{1},s^{-},u_{2},v_{2}),u_{4},v_{4})\\
&  =\left[
\begin{array}
[c]{cc}%
e^{0} & 0\\
0 & e^{\tau_{3}\mu_{2}(u_{4},v_{4})}%
\end{array}
\right]  \left[
\begin{array}
[c]{c}%
e^{\tau_{1}\mu_{1}(u_{2},v_{2})}\\
e^{\tau_{1}\mu_{2}(u_{2},v_{2})}%
\end{array}
\right]  s^{-}=e^{\tau_{1}\mu_{1}(u_{2},v_{2})}s^{-}.
\end{align*}
Thus also the second equality in (\ref{R1}) holds with $\sigma^{-}=\tau
_{3}+\tau_{1}$ and $\alpha^{-}=e^{\tau_{1}\mu_{1}(u_{2},v_{2})}<1$. Now
Theorem \ref{Theorem_cones} implies that there are four control set in
$\mathbb{R}^{2}$ given by the interiors of the four quadrants.

Observe that conditions (\ref{Ex1_A1}), (\ref{Ex1_A2}), and (\ref{Ex1_A3}) are
satisfied in the simple example with $A(u,v)=\left[
\begin{array}
[c]{cc}%
u & 0\\
0 & v
\end{array}
\right]  $ and $\Omega=[-1,1]\times\lbrack-1,1]$. Then $\mu_{1}(u,v)=u,~\mu
_{2}(u,v)=v$, and one may choose%
\[
(u_{1},v_{1})=(1,-1),~(u_{2},v_{2})=(-1,1),~(u_{3},v_{3})=(0,1),~(u_{4}%
,v_{4})=(0,-1).
\]

\end{example}

The next example shows that the situation is quite different if $A$ is a
two-dimensional Jordan block; in particular, scalar controls suffice to verify
assumption (\ref{R1}) in Theorem \ref{Theorem_cones} for a control set
$_{\mathbb{S}}D\not =\mathbb{S}^{1}$.

\begin{example}
Consider%
\begin{equation}
\left[
\begin{array}
[c]{c}%
\dot{x}\\
\dot{y}%
\end{array}
\right]  =\left(  \left[
\begin{array}
[c]{cc}%
\lambda & 1\\
0 & \lambda
\end{array}
\right]  +u(t)\left[
\begin{array}
[c]{cc}%
b_{11} & b_{12}\\
0 & b_{11}%
\end{array}
\right]  \right)  \left[
\begin{array}
[c]{c}%
x\\
y
\end{array}
\right]  , \label{Exa2}%
\end{equation}
with $\lambda\in\mathbb{R}$ and $u(t)\in\Omega$. The system can be written as%
\[
\left[
\begin{array}
[c]{c}%
\dot{x}\\
\dot{y}%
\end{array}
\right]  =\left[
\begin{array}
[c]{cc}%
\lambda+b_{11}u & 1+b_{12}u\\
0 & \lambda+b_{11}u
\end{array}
\right]  \left[
\begin{array}
[c]{c}%
x\\
y
\end{array}
\right]  =A(u)\left[
\begin{array}
[c]{c}%
x\\
y
\end{array}
\right]  .
\]
For all $u\in\Omega$ the eigenvalue $\mu(u)=\lambda+b_{11}u$ has the
eigenspace $\mathbb{R}\times\{0\}$. The intersection of the unit circle with
the eigenspace is given by $\{(1,0)^{\top},(-1,0)^{\top}\}$, which are fixed
under any control for the projected system. Suppose that $b_{12}\not =0$ and
$\Omega$ contains the two points $u_{1}:=0$ and $u_{2}:=-2/b_{12}$, and write
$\mu_{1}=\mu(u_{1})=\lambda$ and $\mu_{2}=\mu(u_{2})=\lambda-2\frac{b_{11}%
}{b_{12}}$. Thus we consider the two differential equations%
\begin{equation}
\left[
\begin{array}
[c]{c}%
\dot{x}\\
\dot{y}%
\end{array}
\right]  =\left[
\begin{array}
[c]{cc}%
\mu_{1} & 1\\
0 & \mu_{1}%
\end{array}
\right]  \left[
\begin{array}
[c]{c}%
x\\
y
\end{array}
\right]  \text{ and }\left[
\begin{array}
[c]{c}%
\dot{x}\\
\dot{y}%
\end{array}
\right]  =-\left[
\begin{array}
[c]{cc}%
-\mu_{2} & 1\\
0 & -\mu_{2}%
\end{array}
\right]  \left[
\begin{array}
[c]{c}%
x\\
y
\end{array}
\right]  . \label{Exa3}%
\end{equation}
The solutions of (\ref{Exa3}) are given by%
\[
\psi_{1}(t,x_{0},y_{0})=e^{\mu_{1}t}\left[
\begin{array}
[c]{c}%
x_{0}+ty_{0}\\
y_{0}%
\end{array}
\right]  ,\quad\psi_{2}(t,x_{0},y_{0})=e^{\mu_{2}t}\left[
\begin{array}
[c]{c}%
x_{0}-ty_{0}\\
y_{0}%
\end{array}
\right]  ,
\]
resp. For the projected systems on the unit circle the trajectory on the upper
half-plane of the first equation tends for $t\rightarrow\infty$ to $(1,0)$ and
for $t\rightarrow-\infty$ to $(-1,0)$. The trajectory for the second equation
moves in the opposite direction. This proves that the open upper semicircle on
$\mathbb{S}^{1}$ is an invariant control set $_{\mathbb{S}}D_{1}$.
Analogously, also the open lower semicircle on $\mathbb{S}^{1}$ is an
invariant control set $_{\mathbb{S}}D_{2}$.

In order to verify the conditions in (\ref{R1}) fix a point $s^{+}%
\in\,_{\mathbb{S}}D_{1}$. Let $\tau>0$ and define%
\[
u^{+}(t)=\left\{
\begin{array}
[c]{lll}%
u_{1} & \text{for} & t\in\lbrack0,\tau]\\
u_{2} & \text{for} & t\in(\tau,2\tau]
\end{array}
\right.  .
\]
It follows that%
\[
\varphi(2\tau,s^{+},u^{+})=\psi_{2}(\tau,\psi_{1}(\tau,s^{+}))=e^{\mu_{2}%
\tau+\mu_{1}\tau}s^{+}.
\]
Then $\alpha^{+}=e^{\mu_{2}\tau+\mu_{1}\tau}>1$ if and only if $\mu_{2}%
+\mu_{1}=2\lambda-2\frac{b_{11}}{b_{12}}>0$, i.e., $\lambda>\frac{b_{11}%
}{b_{12}}$. Similarly, we can find conditions for $\alpha^{-}<1$: The control
sets on the unit sphere do not change if we add a third control value $u_{3}$
which will be specified in a moment. Repeating the derivation above, we find
with $\mu_{3}:=\mu(u_{3})$ that $\alpha^{-}:=e^{\mu_{3}\tau+\mu_{2}\tau}<1$ if
and only if $\mu_{3}+\mu_{2}=\lambda+b_{11}u_{3}+\lambda-2\frac{b_{11}}%
{b_{12}}<0$. This is equivalent to
\begin{equation}
u_{3}b_{11}<2\frac{b_{11}}{b_{12}}-2\lambda. \label{cond1}%
\end{equation}
We conclude that condition (\ref{R1}) holds if $\lambda>\frac{b_{11}}{b_{12}}$
for $\Omega=\{u_{1},u_{2},u_{3}\}$ with $u_{1}=0$ and $u_{2}=-\frac{2}{b_{12}%
}$, and $u_{3}$ satisfying (\ref{cond1}). Then there are two invariant control
sets with nonvoid interior in $\mathbb{R}^{2}$ given by the open upper and
lower half-planes. Observe that these conditions hold, e.g., for
\[
\lambda=1,b_{11}=1,b_{12}=2\text{, and }u_{1}=0,u_{2}=-1,u_{3}<-1.
\]

\end{example}

Next we impose stronger assumptions on the homogeneous bilinear control system
(\ref{bilinear_h}). We require that the control range $\Omega$ is a compact
and convex neighborhood of the origin and that the accessibility rank
condition holds on all of $\mathbb{P}^{n-1}$,%
\begin{equation}
\dim\mathcal{LA}\{\mathbb{P}h(u,\cdot);u\in\Omega\}(p)=n-1\text{ for all }%
p\in\mathbb{P}^{n-1}. \label{ARC_P}%
\end{equation}
Then by Colonius and Kliemann \cite[Theorem 7.1.1]{ColK00} there are $k_{0}$
control sets with nonvoid interior in $\mathbb{P}^{n-1}$ denoted by
$_{\mathbb{P}}D_{1},\ldots,\,_{\mathbb{P}}D_{k_{0}},1\leq k_{0}\leq n$.
Exactly one of these control sets is an invariant control set.

\begin{remark}
Braga Barros and San Martin \cite{BraSM96} use the classification of
semisimple Lie groups acting transitively on projective space $\mathbb{P}%
^{n-1}$ (cf. Boothby and Wilson \cite{BW}) to determine the number $k_{0}%
\in\{1,\ldots,n\}$ of control sets $_{\mathbb{P}}D_{i}$ in projective space
(it is either equal to $n$, $n/2$, or $n/4$).
\end{remark}

Next we analyze the relations between the control sets for the induced systems
on projective space $\mathbb{P}^{n-1}$ and on the unit sphere\ $\mathbb{S}%
^{n-1}$. We will frequently use the following elementary facts that follow
from (\ref{homogene}):

Let $s_{1},s_{2}\in\mathbb{S}^{n-1}$. If $s_{2}$ can be reached from $s_{1}$
(for system (\ref{sphere})), then $-s_{2}$ can be reached from $-s_{1}$. If on
$\mathbb{P}^{n-1}$ the point $\mathbb{P}s_{2}$ can be reached from
$\mathbb{P}s_{1}$, then on $\mathbb{S}^{n-1}$ at least one of the points
$s_{2}$ or $-s_{2}$ can be reached from $s_{1}$.

The proof of the following lemma is modeled after Bacciotti and Vivalda
\cite[Lemma 3]{BacV13}, where controllable systems are analyzed.

\begin{lemma}
\label{Lemma_sphere}(i) Let $_{\mathbb{S}}D$ be a control set on
$\mathbb{S}^{n-1}$. Then the projection of $_{\mathbb{S}}D$ to $\mathbb{P}%
^{n-1}$ is contained in a control set $_{\mathbb{P}}D$.

(ii) Assume that the accessibility rank condition (\ref{ARC_P}) on
$\mathbb{P}^{n-1}$ holds and consider a control set $_{\mathbb{P}}D_{i}$ on
$\mathbb{P}^{n-1}$. Suppose that there is $s_{0}\in\mathbb{S}^{n-1}$ such that
$\mathbb{P}s_{0}\in\mathrm{int}\left(  _{\mathbb{P}}D_{i}\right)  $ and
$-s_{0}$ can be reached from $s_{0}$. Then there exists a control set
$_{\mathbb{S}}D$ on $\mathbb{S}^{n-1}$ containing $A:=\{s\in\mathbb{S}%
^{n-1}\left\vert \mathbb{P}s\in\,_{\mathbb{P}}D_{i}\right.  \}$.
\end{lemma}

\begin{proof}
Assertion (i) is immediate from the definitions. Concerning assertion (ii) it
is clear that for all $s\in A$ there is a control $u$ such that the trajectory
of system (\ref{sphere}) remains in $A$ for all $t\geq0$. Now let $s_{1}%
,s_{2}\in A$. We have to show that $s_{2}$ is in the closure of the reachable
set $\mathcal{O}^{+}(s_{1})$ for system (\ref{sphere}). Since $\mathbb{P}%
s_{1},\mathbb{P}s_{2}\in\,_{\mathbb{P}}D_{i}$ it follows that $s_{2}%
\in\overline{\mathcal{O}^{+}(s_{1})}$ or $-s_{2}\in\overline{\mathcal{O}%
^{+}(s_{1})}$. In the first case we are done. In the second case it follows
that $s_{2}\in\overline{\mathcal{O}^{+}(-s_{1})}$, and that, by our
assumption, $s_{0}\in\mathcal{O}^{+}(-s_{0})$. As noted in Section
\ref{Section2}, $\mathbb{P}(-s_{1})=\mathbb{P}s_{1}\in\,_{\mathbb{P}}D_{i}$
and $\mathbb{P}s_{0}\in\mathrm{int}\left(  _{\mathbb{P}}D_{i}\right)  $ imply
that $\mathbb{P}s_{0}$ can be reached from $\mathbb{P}(-s_{1})$, hence
$s_{0}\in\mathcal{O}^{+}(-s_{1})$ or $-s_{0}\in\mathcal{O}^{+}(-s_{1})$. We
claim that also in the second case one can reach $s_{0}$ from $-s_{1}$. In
fact, $-s_{0}\in\mathcal{O}^{+}(-s_{1})$ implies $s_{0}\in\mathcal{O}%
^{+}(-s_{0})\subset\mathcal{O}^{+}(-s_{1})$. Hence $s_{0}\in\mathcal{O}%
^{+}(-s_{1})$ and we find $-s_{0}\in\mathcal{O}^{+}(s_{1})$.

Since $\mathbb{P}s_{0},\mathbb{P}s_{2}\in\,_{\mathbb{P}}D$ it follows that
$s_{2}\in\overline{\mathcal{O}^{+}(s_{0})}$ or $-s_{2}\in\overline
{\mathcal{O}^{+}(s_{0})}$. In the second case, $s_{2}\in\overline
{\mathcal{O}^{+}(-s_{0})}\subset\overline{\mathcal{O}^{+}(s_{1})}$ and in the
first case, one has%
\[
s_{2}\in\overline{\mathcal{O}^{+}(s_{0})}\subset\overline{\mathcal{O}%
^{+}(-s_{0})}\subset\overline{\mathcal{O}^{+}(s_{1})}.
\]

\end{proof}

The proof of the next proposition uses arguments from Bacciotti and Vivalda
\cite[Proposition 2]{BacV13}.

\begin{proposition}
\label{Proposition3.11}If accessibility rank condition (\ref{ARC_P}) holds for
the induced system on $\mathbb{P}^{n-1}$, it also holds for the induced system
on $\mathbb{S}^{n-1}$.
\end{proposition}

\begin{proof}
Recall that $\mathbb{P}^{n-1}=(\mathbb{R}^{n}\setminus\{0\})/\thicksim$, where
$\thicksim$ is the equivalence relation $x\thicksim y$ if $y=\lambda x$ with
some $\lambda\not =0$. Furthermore, an atlas of $\mathbb{P}^{n-1}$ is given by
$n$ charts $(U_{i},\psi_{i})$, where $U_{i}$ is the set of equivalence classes
$[x_{1}:\cdots:x_{n}]$ with $x_{i}\not =0$ (the homogeneous coordinates) and
$\psi_{i}:U_{i}\rightarrow\mathbb{R}^{n-1}$ is defined by%
\[
\psi_{i}([x_{1}:\cdots:x_{n}])=\left(  \frac{x_{1}}{x_{i}},\ldots,\frac
{\hat{x}_{i}}{x_{i}},\ldots,\frac{x_{n}}{x_{i}}\right)  ,
\]
where the hat means that the $i$-th entry is missing.

For the sake of simplicity we prove the rank condition for the North Pole of
$\mathbb{S}^{n-1}$ given by $\bar{z}_{0}=(0,\ldots,0,1)$. By assumption, the
rank of the Lie algebra of the system on $\mathbb{P}^{n-1}$ is $n-1$ on all of
$\mathbb{P}^{n-1}$. Consider the point $x_{0}=[0:\cdots:0:1]\in\mathbb{P}%
^{n-1}$. Thus there exist $n-1$ matrices $A_{1},\ldots,A_{n-1}$ in the Lie
algebra generated by the system on $\mathbb{R}^{n}\setminus\{0\}$ such that
for the induced vector fields $A_{1}^{\flat},\ldots,A_{n-1}^{\flat}$ in the
Lie algebra for the system on $\mathbb{P}^{n-1}$ one obtains that the rank of
the family $\left(  A_{1}^{\flat}(x_{0}),\ldots,A_{n-1}^{\flat}(x_{0})\right)
$ is $n-1$. Now \cite[formula (5)]{BacV13} shows the following formula for the
local expression of this family, which has the form $\left(  A_{1}^{n}%
(z_{0}),\ldots,A_{n-1}^{n}(z_{0})\right)  $ with $z_{0}=(0,\ldots,0)$; let
$a_{1}^{k}(\bar{z}_{0}),\ldots,a_{n}^{k}(\bar{z}_{0})$ denote the $n$
components of $A_{k}\bar{z}_{0}$. Then, for $k=1,\ldots,n-1$,
\[
A_{k}^{n}(z_{0})=(a_{1}^{k}(\bar{z}_{0}),\ldots,a_{n-1}^{k}(\bar{z}%
_{0}))^{\top}-a_{n}^{k}(\bar{z}_{0})z_{0}=(a_{1}^{k}(\bar{z}_{0}%
),\ldots,a_{n-1}^{k}(\bar{z}_{0}))^{\top}.
\]
So $A_{k}^{n}(z_{0})$ is the vector whose components are equal to the first
$n-1$ components of the last column of the matrix $A_{k}$.

On the other hand, the projections on $\mathbb{S}^{n-1}$ of the linear vector
fields for the matrices $A_{1},\ldots,A_{n-1}$ are the vector fields (cf.
(\ref{sphere}))
\[
A_{k}^{\circ}(x)=A_{k}x-x^{\top}Ax\cdot x,\quad x\in\mathbb{S}^{n-1}.
\]
Thus we get, for $k=1,\ldots,n-1$%
\[
A_{k}^{\circ}(\bar{z}_{0})=A_{k}\bar{z}_{0}-\bar{z}_{0}^{\top}A_{k}\bar{z}%
_{0}\cdot\bar{z}_{0}=(a_{1}^{k}(\bar{z}_{0}),\ldots,a_{n-1}^{k}(\bar{z}%
_{0}),a_{n}^{k}(z_{0})-\bar{z}_{0}^{\top}A_{k}\bar{z}_{0})^{\top},
\]
so the $n-1$ first components of $A_{k}^{\circ}(\bar{z}_{0})$ are equal to the
components of $A_{k}^{n}(z_{0})$. This implies that the vectors $A_{1}^{\circ
}(\bar{z}_{0}),\ldots,A_{n-1}^{\circ}(\bar{z}_{0})$ are linearly independent.
\end{proof}

We get the following result characterizing the relation between the control
sets $_{\mathbb{P}}D_{1},\ldots,\,_{\mathbb{P}}D_{k_{0}},1\leq k_{0}\leq n$,
on projective space and the control sets on the unit sphere.

\begin{teo}
\label{Theorem_sphere}Suppose that accessibility rank condition (\ref{ARC_P})
holds for the induced system on projective space $\mathbb{P}^{n-1}$.

(i) If there is $s_{0}\in\mathbb{S}^{n-1}$ with $\mathbb{P}s_{0}%
\in\mathrm{int}\left(  _{\mathbb{P}}D_{i}\right)  $ such that $-s_{0}$ can be
reached for system (\ref{sphere}) from $s_{0}$, then $_{\mathbb{S}}%
D:=\{s\in\mathbb{S}^{n-1}\left\vert \mathbb{P}s\in\,_{\mathbb{P}}D_{i}\right.
\}$ is the unique control set on $\mathbb{S}^{n-1}$ which projects to
$_{\mathbb{P}}D_{i}$.

(ii) For every control set $_{\mathbb{P}}D_{i},i\in\{1,\ldots,k_{0}\}$, there
are at most two control sets $_{\mathbb{S}}D$ and $_{\mathbb{S}}D^{\prime}$ on
$\mathbb{S}^{n-1}$ with nonvoid interior such that%
\begin{equation}
\{s\in\mathbb{S}^{n-1}\left\vert \mathbb{P}s\in\,_{\mathbb{P}}D_{i}\right.
\}=\,_{\mathbb{S}}D\cup\,_{\mathbb{S}}D^{\prime}, \label{sphere_1}%
\end{equation}
and $_{\mathbb{S}}D=-$\thinspace$_{\mathbb{S}}D^{\prime}$.

(iii) There are $k_{1}$ control sets with nonvoid interior on $\mathbb{S}%
^{n-1}$ denoted by $_{\mathbb{S}}D_{1},\ldots,\allowbreak\,_{\mathbb{S}%
}D_{k_{1}}$ with $1\leq k_{1}\leq2k_{0}\leq2n$. At most two of the sets
$_{\mathbb{S}}D_{i}$ are invariant control sets.
\end{teo}

\begin{proof}
(i) Suppose that there is $\mathbb{P}s_{0}\in\mathrm{int}\left(  _{\mathbb{P}%
}D_{i}\right)  $ with $-s_{0}\in\mathcal{O}^{+}(s_{0})$. By Lemma
\ref{Lemma_sphere}(ii) there is a control set $_{S}D$ on the unit sphere
containing $\{s\in\mathbb{S}^{n-1}\left\vert \mathbb{P}s\in\,_{\mathbb{P}%
}D_{i}\right.  \}$, hence the projection of $_{\mathbb{S}}D$ to projective
space contains $_{\mathbb{P}}D_{i}$. Using Lemma \ref{Lemma_sphere}(i) one
concludes that $_{\mathbb{S}}D=\{s\in\mathbb{S}^{n-1}\left\vert \mathbb{P}%
s\in\,_{\mathbb{P}}D_{i}\right.  \}$.

(ii) Fix a point $s_{0}\in\mathbb{S}^{n-1}$ with $\mathbb{P}s_{0}%
\in\mathrm{int}\left(  _{\mathbb{P}}D_{i}\right)  $ and define%
\begin{align*}
A^{+}  &  :=\left\{  s\in\mathbb{S}^{n-1}\left\vert \,\mathbb{P}%
s\in\,_{\mathbb{P}}D_{i}\text{ and }s\in\mathcal{O}^{+}(s_{0})\cap
\mathcal{O}^{-}(s_{0})\right.  \right\}  ,\\
A^{-}  &  :=\left\{  s\in\mathbb{S}^{n-1}\left\vert \,\mathbb{P}%
s\in\,_{\mathbb{P}}D_{i}\text{ and }-s\in\mathcal{O}^{+}(s_{0})\cap
\mathcal{O}^{-}(s_{0})\right.  \right\}  .
\end{align*}
The set $A^{+}$ is contained in a control set $_{\mathbb{S}}D$ and the set
$A^{-}$ is contained in a control set $_{\mathbb{S}}D^{\prime}$. Every point
$s$ with $\mathbb{P}s\in\mathrm{int}\left(  _{\mathbb{P}}D_{i}\right)  $
satisfies $s\in\mathcal{O}^{+}(s_{0})$ or $-s\in\mathcal{O}^{+}(s_{0})$ and it
also satisfies $s\in\mathcal{O}^{-}(s_{0})$ or $-s\in\mathcal{O}^{-}(s_{0})$.
If there is $s\in\mathcal{O}^{+}(s_{0})$ with $-s\in\mathcal{O}^{-}(s_{0})$
hence $s\in\mathcal{O}^{-}(-s_{0})$, it follows $-s_{0}\in\mathcal{O}%
^{+}(s_{0})$. Then by part (i) the assertion follows. The same arguments apply
if there is $s$ with $-s\in\mathcal{O}^{+}(s_{0})$ and $s\in\mathcal{O}%
^{-}(s_{0})$. Hence we may assume that either $s\in\mathcal{O}^{+}(s_{0}%
)\cap\mathcal{O}^{-}(s_{0})$ or $-s\in\mathcal{O}^{+}(s_{0})\cap
\mathcal{O}^{-}(s_{0})$. This shows that
\[
\{s\in\mathbb{S}^{n-1}\left\vert \mathbb{P}s\in\mathrm{int}\left(
_{\mathbb{P}}D_{i}\right)  \right.  \}\subset A^{+}\cup A^{-}\subset\left(
_{\mathbb{S}}D\right)  \cup\left(  _{\mathbb{S}}D^{\prime}\right)  .
\]
It follows that $\{s\in\mathbb{S}^{n-1}\left\vert \mathbb{P}s\in
\,_{\mathbb{P}}D_{i}\right.  \}\subset\overline{_{\mathbb{S}}D}\cup
\overline{_{\mathbb{S}}D^{\prime}}$, since $\mathbb{P}$ is an open map and
$_{\mathbb{P}}D_{i}\subset\overline{\mathrm{int}\left(  _{\mathbb{P}}%
D_{i}\right)  }$. By Lemma \ref{Lemma_sphere}(i) the projections of
$_{\mathbb{S}}D$ and $_{\mathbb{S}}D^{\prime}$ to $\mathbb{P}^{n-1}$ are
contained in $_{\mathbb{P}}D_{i}$, hence (\ref{sphere_1}) follows. The same
arguments with $-s_{0}$ instead of $s_{0}$ implies that $_{\mathbb{S}}%
D=-$\thinspace$_{\mathbb{S}}D^{\prime}$. If $_{\mathbb{S}}D$ or $_{\mathbb{S}%
}D^{\prime}$ is an invariant control set, then also $_{\mathbb{P}}D_{i}$ is an
invariant control set, hence there are at most two invariant control set on
$\mathbb{S}^{n-1}$.

(iii) This is a consequence of assertion (ii).
\end{proof}

Recall the following definitions from Colonius and Kliemann \cite{ColK00}. For
a solution $\varphi(t,x,u),t\geq0$, of (\ref{bilinear_h}) the Lyapunov
exponent is%
\begin{equation}
\lambda(u,x)=\underset{t\rightarrow\infty}{\lim\sup}\frac{1}{t}\log\left\Vert
\varphi(t,x,u)\right\Vert . \label{Lyap}%
\end{equation}
Observe that the Lyapunov exponents are constant on lines through the origin.

\begin{definition}
For a control set $_{\mathbb{P}}D$ in $\mathbb{P}^{n-1}$ the Floquet spectrum
is given by%
\[
\Sigma_{Fl}(_{\mathbb{P}}D)=\left\{  \lambda(u,x)\left\vert
\begin{array}
[c]{c}%
\mathbb{P}x\in\mathrm{int}\left(  _{\mathbb{P}}D\right)  \text{ and }u\text{
is piecewise constant}\\
\tau\text{-periodic for some }\tau\geq0\text{ with }\mathbb{P}\varphi
(\tau,x,u)=\mathbb{P}x
\end{array}
\right.  \right\}  ,
\]
and for a control set $_{\mathbb{S}}D$ in $\mathbb{S}^{n-1}$ the Floquet
spectrum is given by%
\[
\Sigma_{Fl}(_{\mathbb{S}}D)=\left\{  \lambda(u,x)\left\vert
\begin{array}
[c]{c}%
x\in\mathrm{int}\left(  _{\mathbb{S}}D\right)  \text{ and }u\text{ is
piecewise constant}\\
\tau\text{-periodic for some }\tau\geq0\text{ with }s(\tau,x,u)=x
\end{array}
\right.  \right\}  .
\]

\end{definition}

In the $\tau$-periodic case considered here the Lyapunov exponents satisfy
$\lambda(u,x)=\frac{1}{\tau}\log\left\Vert \varphi(\tau,x,u)\right\Vert $ for
$\left\Vert x\right\Vert =1$ and coincide with the Floquet exponents (cf.
Teschl \cite[\S 3.6]{Tes}). We note the following result.

\begin{proposition}
\label{Proposition_Floquet}If $_{S}D$ is a control set with nonvoid interior
on $\mathbb{S}^{n-1}$ that projects to a control set $_{\mathbb{P}}D$ in
$\mathbb{P}^{n-1}$, then%
\[
\Sigma_{Fl}(_{\mathbb{S}}D)=\Sigma_{Fl}(_{\mathbb{P}}D).
\]

\end{proposition}

\begin{proof}
The inclusion \textquotedblleft$\Sigma_{Fl}(_{\mathbb{S}}D)\subset\Sigma
_{Fl}(_{\mathbb{P}}D)$\textquotedblright\ is clear. For the converse, consider
$\mathbb{P}x\in\mathrm{int}\left(  _{\mathbb{P}}D\right)  $ and a piecewise
constant $\tau$-periodic control $u$ with $\mathbb{P}\varphi(\tau
,x,u)=\mathbb{P}x$. We may suppose that $x\in\mathbb{S}^{n-1}$, hence
$x\in\,_{\mathbb{S}}D$ or $-x\in\,_{\mathbb{S}}D$. Consider the first case. If
$\varphi(\tau,x,u)=\alpha x$ with $\alpha>0$ it follows that $\lambda
(u,x)=\frac{1}{\tau}\log\alpha\in\Sigma_{Fl}(_{\mathbb{S}}D)$. Otherwise
$\varphi(\tau,x,u)=-\alpha x$ with $\alpha>0$ and hence%
\[
\varphi(2\tau,x,u)=\varphi(\tau,\varphi(\tau,x,u),u(\tau+\cdot))=-\alpha
\left(  -\alpha x\right)  =\alpha^{2}x,
\]
implying
\[
\lambda(u,x)=\frac{1}{2\tau}\log\left\Vert \varphi(2\tau,x,u)\right\Vert
=\frac{1}{2\tau}\log\alpha^{2}=\frac{1}{\tau}\log\alpha\in\Sigma
_{Fl}(_{\mathbb{S}}D).
\]
Analogously one argues in the case $-x\in\,_{\mathbb{S}}D$.
\end{proof}

The following result describes the control sets in $\mathbb{R}^{n}$ under the
accessibility rank condition on projective space.

\begin{teo}
\label{Theorem_95}Assume that the homogeneous bilinear control system
(\ref{bilinear_h}) satisfies the accessibility rank condition (\ref{ARC_P}) on
$\mathbb{P}^{n-1}$. If a control set $_{\mathbb{S}}D_{i},i\in\{1,\ldots
,k_{1}\}$, on $\mathbb{S}^{n-1}$ satisfies $0\in\mathrm{int}\left(
\Sigma_{Fl}(_{\mathbb{S}}D_{i})\right)  $, then the cone%
\[
D_{i}=\{\alpha x\in\mathbb{R}^{n}\left\vert \alpha>0\text{ and }%
x\in\,_{\mathbb{S}}D_{i}\right.  \}
\]
generated by $_{\mathbb{S}}D_{i}$ is a control set with nonvoid interior in
$\mathbb{R}^{n}\setminus\{0\}$. At most two of the $D_{i}$ are invariant
control sets.
\end{teo}

\begin{proof}
By Proposition \ref{Proposition3.11}, every point in $_{\mathbb{S}}D_{i}$ is
locally accessible. Hence the first assertion follows from Theorem
\ref{Theorem_cones}, if we can show that assumption (ii) in that theorem
holds. The Floquet spectrum over a control set in projective space is a
bounded interval, cf. \cite[Proposition 6.2.14]{ColK00}. By Proposition
\ref{Proposition_Floquet} the same holds true for the Floquet spectrum of
$\Sigma_{Fl}(_{\mathbb{S}}D_{i})$. If $0\in\mathrm{int}\left(  \Sigma
_{Fl}(_{\mathbb{S}}D_{i})\right)  $, it follows that there are points
$s^{+},s^{-}\in\mathrm{int}\left(  _{\mathbb{S}}D_{i}\right)  $, controls
$u^{+},u^{-}\in\mathcal{U}$ and times $\sigma^{+},\sigma^{-}>0$ such that
\[
\varphi(\sigma^{+},s^{+},u^{+})=\alpha^{+}s^{+},\quad\varphi(\sigma^{-}%
,s^{-},u^{-})=\alpha^{-}s^{-},
\]
where $\alpha^{+}:=\exp(\sigma^{+}\lambda(u^{+},s^{+}))\in(1,\infty)$ and
$\alpha^{-}:=\exp(\sigma^{-}\lambda(u^{-},s^{-}))\in(0,1)$. This verifies
assumption (ii) of Theorem \ref{Theorem_cones} if we take into account that we
may vary $\sigma^{+}$ and hence $\alpha^{+}$. Furthermore, every invariant
control set $D$ projects to an invariant control set on $\mathbb{S}^{n-1}$,
and here there are at most two invariant control sets.
\end{proof}

\begin{remark}
Theorem \ref{Theorem_95} corrects Colonius and Kliemann \cite[Corollary
12.2.6]{ColK00}, \cite[Theorem 7]{ColK95}, where, assuming the stronger
accessibility rank condition in $\mathbb{R}^{n}\setminus\{0\}$, a similar
statement was given \ However, it was not taken into account that there may
exist two control sets on the unit sphere that project to the same control set
on projective space. It remains an open question if there are control sets
$_{\mathbb{S}}D_{i}$ with $0\not \in \mathrm{int}\left(  \Sigma_{Fl}%
(_{\mathbb{S}}D_{i})\right)  $ that generate cones which are control sets on
$\mathbb{R}^{n}\setminus\{0\}$.
\end{remark}

\begin{remark}
\label{Remark_inv_cone}Suppose that under the assumptions of Theorem
\ref{Theorem_95} an invariant control set $D_{i}$ in $\mathbb{R}^{n}%
\setminus\{0\}$ exists. Then $D_{i}\cup\{0\}$ is a closed cone in
$\mathbb{R}^{n}$ generated by an invariant control set on the unit sphere. If
the system is not controllable, this cone does not coincide with
$\mathbb{R}^{n}$, hence it is a nontrivial proper closed positively invariant
cone in $\mathbb{R}^{n}$. On the other hand, Do Rocio, San Martin, and Santana
\cite[Section 6]{DoRoSMS06} present an example in $\mathbb{R}^{4}$, which is
not controllable and which also does not possess a nontrivial proper closed
\emph{convex} cone $W$ in $\mathbb{R}^{n}$ which is positively invariant. Here
the convexity of $W$ is crucial: Such cones are pointed, i.e., $W\cap
(-W)=\{0\}$, cf. \cite[Lemma 4.1]{DoRoSMS06}. For an invariant control set
$D_{i}$ as in Theorem \ref{Theorem_95} the cone $D_{i}\cup\{0\}$ need not be
pointed (and hence not convex), since the invariant control set may contain
the real eigenspace for a complex conjugate pair of eigenvalues of $A(u)$.
Observe that here the convex closure of this cone, which is also positively
invariant, coincides with $\mathbb{R}^{n}$. An example is the
three-dimensional linear oscillator in Colonius and Kliemann \cite[Example
10.2.3]{ColK00}. The existence of nontrivial proper closed convex positively
invariant cones in $\mathbb{R}^{n}$ is analyzed in \cite[Theorem 4.2, Theorem
4.5]{DoRoSMS06}.
\end{remark}

Not all control sets on the unit sphere generate cones that are control sets
in $\mathbb{R}^{n}\setminus\{0\}$ as indicated by the following proposition,

\begin{proposition}
Assume that the homogeneous bilinear control system (\ref{bilinear_h})
satisfies the accessibility rank condition (\ref{ARC_P}) on $\mathbb{S}^{n-1}$
and let $_{\mathbb{S}}D$ be a control set in $\mathbb{S}^{n-1}$ with nonvoid
interior. Then the following assertion holds.

If the supremum of $\left\{  \lambda(u,x)\left\vert s(t,x,u)\in\,_{\mathbb{S}%
}D\text{ for all }t\geq0\right.  \right\}  $ is less than $0$ or the infimum
is greater than $0$, then the cone%
\[
C=\{\alpha x\in\mathbb{R}^{n}\left\vert \alpha>0\text{ and }x\in
\,_{\mathbb{S}}D\right.  \}
\]
is not a control set.
\end{proposition}

\begin{proof}
Exact controllability to points in the interior of $_{\mathbb{S}}D$ implies
that for all $x,y\in\mathbb{R}^{n}\setminus\{0\}$ with $\frac{x}{\left\Vert
x\right\Vert }\in\mathrm{\,}_{\mathbb{S}}D$ and $\frac{y}{\left\Vert
y\right\Vert }\in\mathrm{int}(_{\mathbb{S}}D)$ there are $\alpha,T>0$ and
$u\in\mathcal{U}$ with $\varphi(T,x,u)=\alpha y$. Now consider $(x,u)$ with
$s(t,x,u)\in\,_{\mathbb{S}}D$ for all $t\geq0$. Then in the first case the
trajectory in $\mathbb{R}^{n}$ satisfies $\left\Vert \varphi(t,x,u\right\Vert
\rightarrow0$ and in the second case it satisfies $\left\Vert \varphi
(t,x,u\right\Vert \rightarrow\infty$. Hence the assertion follows.
\end{proof}

\begin{remark}
We refer to Colonius and Kliemann \cite{ColK00} for a discussion when the
supremum and the infimum of $\left\{  \lambda(u,x)\left\vert s(t,x,u)\in
\,_{\mathbb{S}}D\text{ for all }t\geq0\right.  \right\}  $ coincide with the
supremum and the infimum of $\Sigma_{Fl}(_{\mathbb{S}}D)$, respectively. For
dimension $n=2$, \cite[Theorem 10.1.1]{ColK00} shows that these equalities
hold if the accessibility rank condition holds in $\mathbb{P}^{1}$. For
general $n\in\mathbb{N}$ suppose that the control range is given by $\rho
\cdot\Omega,\rho\geq0$, and the following \textquotedblleft$\rho$-inner-pair
condition" for the system on $\mathbb{S}^{n-1}$ holds:%
\[
\text{For all }\rho^{\prime}>\rho\text{ every }(u,x)\in\mathcal{U}%
\times\mathbb{S}^{n-1}\text{ there is }t>0\text{ with }s(t,x,u)\in
\mathrm{int}(\mathcal{O}^{+}(x)).
\]
Then \cite[Theorem 7.3.26]{ColK00} implies that for all $\rho\in(0,\infty)$
except for at most $n-1$ $\rho$-values the systems with control range
$\rho\cdot\Omega$ have the property that the equalities for the suprema and
the infima hold for all control sets.
\end{remark}

The following example illustrates Theorem \ref{Theorem_95}, cf. also
\cite[Examples 10.1.7 and 10.2.1]{ColK00} where for linear oscillators the
spectral properties and the control sets in projective space are determined.

\begin{example}
\label{Example6}Consider the damped linear oscillator%
\[
\ddot{x}+3\dot{x}+(1+u(t))x=0\text{ with }u(t)\in\Omega=[-\rho,\rho],
\]
where $\rho\in\left(  1,\frac{5}{4}\right)  $. Hence the system equation is
given by
\begin{equation}
\left[
\begin{array}
[c]{c}%
\dot{x}\\
\dot{y}%
\end{array}
\right]  =\left(  \left[
\begin{array}
[c]{cc}%
0 & 1\\
-1 & -3
\end{array}
\right]  +u\left[
\begin{array}
[c]{cc}%
0 & 0\\
-1 & 0
\end{array}
\right]  \right)  \left[
\begin{array}
[c]{c}%
x\\
y
\end{array}
\right]  =\left[
\begin{array}
[c]{cc}%
0 & 1\\
-1-u & -3
\end{array}
\right]  \left[
\begin{array}
[c]{c}%
x\\
y
\end{array}
\right]  . \label{Example6_a}%
\end{equation}
The eigenvalues of $A(u)$ satisfy%
\[
\det(\lambda I-A(u))=\det\left(
\begin{array}
[c]{cc}%
\lambda & -1\\
1+u & \lambda+3
\end{array}
\right)  =\lambda^{2}+3\lambda+1+u=0,
\]
and one obtains two real eigenvalues%
\[
\lambda_{1}(u)=-\frac{3}{2}-\sqrt{\frac{5}{4}-u}\text{ and }\lambda
_{2}(u)=-\frac{3}{2}+\sqrt{\frac{5}{4}-u}%
\]
with corresponding eigenvectors $(x,\lambda_{1}(u)x)^{\top}$ and
$(x,\lambda_{2}(u)x)^{\top},x\not =0$. Note that $\lambda_{2}(u)>0$ if and
only if $u\in\lbrack-\rho,-1)$. Since for all $u\in\lbrack-\rho,\rho]$ one has
$\lambda_{1}(u)<\lambda_{2}(u)$ the projected trajectories in $\mathbb{P}^{1}$
go from the eigenspace for $\lambda_{1}(u)$ to the eigenspace for $\lambda
_{2}(u)$. A short computation shows that there is an open control set
$_{\mathbb{P}}D_{1}$ and a closed invariant control set $_{\mathbb{P}}D_{2}$
in projective space $\mathbb{P}^{1}$ given by the projections of
\[
\left\{  \left[
\begin{array}
[c]{c}%
x\\
\lambda x
\end{array}
\right]  \left\vert x\not =0,\,\lambda\in\Sigma_{Fl}(_{\mathbb{P}}%
D_{1})\right.  \right\}  ,\quad\left\{  \left[
\begin{array}
[c]{c}%
x\\
\lambda x
\end{array}
\right]  \left\vert x\not =0,\,\lambda\in\overline{\Sigma_{Fl}(_{\mathbb{P}%
}D_{2})}\right.  \right\}  ,
\]
resp., where by \cite[Theorem 10.1.1]{ColK00} the Floquet spectra are%
\begin{align*}
\Sigma_{Fl}(_{\mathbb{P}}D_{1})  &  =\left(  -\frac{3}{2}-\sqrt{\frac{5}%
{4}+\rho},-\frac{3}{2}-\sqrt{\frac{5}{4}-\rho}\right)  \subset(-\infty,0),\\
\Sigma_{Fl}(_{\mathbb{P}}D_{2})  &  =\left(  -\frac{3}{2}+\sqrt{\frac{5}%
{4}-\rho},-\frac{3}{2}+\sqrt{\frac{5}{4}+\rho}\right)  .
\end{align*}
The control sets in $\mathbb{P}^{1}$ induce four control sets on the unit
circle $\mathbb{S}^{1}$. For $_{\mathbb{P}}D_{2}$ one obtains the two control
sets $_{\mathbb{S}}D_{2}^{\prime}=-\,_{\mathbb{S}}D_{2}$. Since $u=-1\in
(-\rho,\rho)$ and $0=\lambda_{2}(-1)\in\mathrm{int}\left(  \Sigma
_{Fl}(_{\mathbb{P}}D_{2})\right)  $, Theorem \ref{Theorem_95} implies that
there are two invariant control sets in $\mathbb{R}^{2}\setminus\{0\}$, they
are the cones%
\[
D_{2}=\left\{  \alpha\left[
\begin{array}
[c]{c}%
x\\
y
\end{array}
\right]  \left\vert \alpha>0,\left[
\begin{array}
[c]{c}%
x\\
y
\end{array}
\right]  \in~_{\mathbb{S}}D_{2}\right.  \right\}  \text{,~}D_{2}^{\prime
}=\left\{  \alpha\left[
\begin{array}
[c]{c}%
x\\
y
\end{array}
\right]  \left\vert \alpha>0,\left[
\begin{array}
[c]{c}%
x\\
y
\end{array}
\right]  \in~_{\mathbb{S}}D_{2}^{\prime}\right.  \right\}  .
\]

\end{example}

Next we present a necessary and sufficient condition for controllability on
$\mathbb{R}^{n}\setminus\{0\}$. The infimal and supremal Lyapunov exponents,
cf. (\ref{Lyap}), are%
\[
\kappa^{\ast}=\inf_{u\in\mathcal{U}}\inf_{x\not =0}\lambda(u,x)\text{ and
}\kappa=\sup_{u\in\mathcal{U}}\sup_{x\not =0}\lambda(u,x),
\]
resp. The following result improves Colonius and Kliemann \cite[Corollary
12.2.6(iii)]{ColK00}, where the accessibility rank condition is assumed in
$\mathbb{R}^{n}\setminus\{0\}$.

\begin{corollary}
\label{Corollary_approximate}Assume that the homogeneous bilinear control
system (\ref{bilinear_h}) satisfies the accessibility rank condition
(\ref{ARC_P}) on $\mathbb{P}^{n-1}$. Then it is controllable in $\mathbb{R}%
^{n}\setminus\{0\}$ if and only if the induced system on $\mathbb{P}^{n-1}$ is
controllable and $\kappa^{\ast}<0<\kappa$.
\end{corollary}

\begin{proof}
Controllability on $\mathbb{R}^{n}\setminus\{0\}$ implies controllability on
$\mathbb{P}^{n-1}$. Furthermore, asymptotic null controllability to
$0\in\mathbb{R}^{n}$, and hence exponential null controllability follows by
\cite[Corollary 12.2.3]{ColK00}. Thus $\kappa^{\ast}<0$ and, by time reversal,
also $\kappa>0$ follows.

Conversely, controllability on $\mathbb{P}^{n-1}$ implies by Bacciotti and
Vivalda \cite[Theorem 1]{BacV13} that $_{\mathbb{S}}D=\mathbb{S}^{n-1}$ is a
control set. By Theorem \ref{Theorem_95}, it follows that $\mathbb{R}%
^{n}\setminus\{0\}$ is a control set. This implies that for every initial
point $x\not =0$ the reachable set $\mathcal{O}^{+}(x)$ is dense in
$\mathbb{R}^{n}\setminus\{0\}$, i.e., approximate controllability holds. For
homogeneous bilinear control systems, Cannarsa and Sigalotti \cite[Theorem
1]{CanS21} shows that approximate controllability implies controllability in
$\mathbb{R}^{n}\setminus\{0\}$. This completes the proof.
\end{proof}

\begin{remark}
The condition $\kappa^{\ast}<0<\kappa$ can be replaced by the requirement that
$0\in\mathrm{int}\left(  \Sigma_{Fl}(\mathbb{P}^{n-1})\right)  =(\kappa^{\ast
},\kappa)$. This follows, since by \cite[Theorem 7.1.5(iv)]{ColK00} the
Floquet spectrum is an interval and satisfies $\overline{\Sigma_{Fl}%
(\mathbb{P}^{n-1})}=[\kappa^{\ast},\kappa]$ if $\mathbb{P}^{n-1}$ is a control set.
\end{remark}

\begin{remark}
For control systems on semisimple Lie groups, San Martin \cite[Proposition
5.6]{SanM93} shows the following result. Let $G\subset Sl(n,\mathbb{R})$ be a
semisimple, connected, and noncompact group acting transitively on
$\mathbb{R}^{n}\setminus\{0\}$ and let $S$ be a semigroup with nonvoid
interior in $G$. Then $S$ is controllable on $\mathbb{R}^{n}\setminus\{0\}$ if
and only if $S$ is controllable in $\mathbb{P}^{n-1}$. In this case
$0\in(\kappa^{\ast},\kappa)=\mathrm{int}\left(  \Sigma_{Fl}(\mathbb{P}%
^{n-1})\right)  $.
\end{remark}

\section{Equilibria of affine systems\label{Section4}}

In the rest of this paper we discuss control sets for affine systems of the
form (\ref{affine}). We begin by analyzing the equilibria.

For each control value $u\in\Omega$, an associated equilibrium point of system
(\ref{affine}) is a state $x_{u}$ that satisfies
\begin{equation}
0=A(u)x_{u}+Cu+d. \label{equilibrium1}%
\end{equation}
If for $u\in\Omega$ there is a solution $x_{u}$ of (\ref{equilibrium1}) and
$\det A(u)=0$, then every point in the nontrivial affine subspace $x_{u}+\ker
A(u)$ is an equilibrium. If there is $u\in\Omega$ with $Cu+d=0$, then equation
(\ref{equilibrium1}) always has the solution $x_{u}=0$. If $\det A(u)\not =0$,
then there exists a unique equilibrium of (\ref{affine}) given by%
\begin{equation}
x_{u}=-A(u)^{-1}[Cu+d]. \label{equilibrium2}%
\end{equation}
The following simple but useful result shows that for constant control $u$ the
phase portrait of the inhomogeneous equation is obtained by shifting the
origin to $x_{u}$.

\begin{proposition}
\label{Proposition_shift}Consider for constant control $u\in\Omega$ a solution
$\varphi(t,x,u),t\geq0$, of the inhomogeneous equation (\ref{affine}) and let
$x_{u}$ be an associated equilibrium. Then $\varphi(t,x,u)-x_{u}$ is a
solution of the homogeneous equation $\dot{x}(t)=A(u)x(t)$ with initial value
$x-x_{u}$.
\end{proposition}

\begin{proof}
We compute%
\[
\frac{d}{dt}\left[  \varphi(t,x,u)-x_{u}\right]  =A(u)\left[  \varphi
(t,x,u)-x_{u}\right]  +A(u)x_{u}+Cu+d=A(u)\left[  \varphi(t,x,u)-x_{u}\right]
.
\]

\end{proof}

The following proposition shows that the affine control system (\ref{affine})
is equivalent to an inhomogeneous bilinear system, if there is $u^{0}\in
\Omega$ with $Cu^{0}+d=0$.

\begin{proposition}
\label{Proposition_inhomogeneous}Suppose that there is $u^{0}\in\Omega$ with
$Cu^{0}+d=0$ and consider%
\begin{equation}
\dot{x}(t)=A(u^{0})x(t)+\sum_{i=1}^{m}v_{i}(t)B_{i}x(t)+Cv(t)\text{ with
}v(t)\in\Omega^{\prime}:=\Omega-u^{0}, \label{affine_hom}%
\end{equation}
with trajectories denoted by $\psi(\cdot,x,v)$. Then the trajectories
$\varphi(\cdot,x,u),u\in\mathcal{U}$, of (\ref{affine}) satisfy $\varphi
(t,x,u)=\psi(t,x,v),t\in\mathbb{R}$, with controls $v(t)=u(t)-u^{0}%
,t\in\mathbb{R}$.
\end{proposition}

\begin{proof}
One computes for a solution $x(t)=\varphi(t,x,u),t\in\mathbb{R}$, of
(\ref{affine})
\begin{align*}
\dot{x}(t)  &  =Ax(t)+\sum_{i=1}^{m}u_{i}^{0}B_{i}x(t)+\sum_{i=1}^{m}\left(
u_{i}(t)-u_{i}^{0}\right)  B_{i}x(t)+C(u(t)-u^{0})+Cu^{0}+d\\
&  =A(u^{0})x(t)+\sum_{i=1}^{m}v_{i}(t)B_{i}x(t)+Cv(t).
\end{align*}

\end{proof}

We introduce the following notation for the set of equilibria,%
\begin{align*}
E  &  =\{x\in\mathbb{R}^{n}\left\vert 0=A(u)x+Cu+d\text{ for some }u\in
\Omega\right.  \},\\
E_{0}  &  =\{x\in\mathbb{R}^{n}\left\vert 0=A(u)x+Cu+d\text{ for some }%
u\in\mathrm{int}\left(  \Omega\right)  \right.  \}\text{.}%
\end{align*}
Note that $\overline{E_{0}}=E$ if $\Omega=\,\overline{\mathrm{int}\left(
\Omega\right)  }$. The following discussion of systems with scalar controls
follows essentially Mohler \cite[Section 2.4]{Mohler}.

\begin{teo}
\label{Theorem2}Consider system (\ref{affine}) with scalar control and assume
that for all $u\in\Omega$ it follows from $\det(A+uB)=0$ that there is no
solution to equation (\ref{equilibrium1}).

(i) Suppose that there is $u^{0}\in\Omega=\mathbb{R}$ with $A+u^{0}B$
nonsingular. Then there are at most $1\leq r\leq n$ control values $v^{i}%
\in\mathbb{R}$ such that the equilibrium set is given by%
\[
E=\{x_{u}\left\vert u\in\mathbb{R}\setminus\{v^{1},\ldots,v^{r}\}\right.  \}
\]
and is the union of at most $n+1$ smooth curves. These curves have no finite endpoints.

(ii) If $\Omega$ is a possibly unbounded interval, the equilibrium set $E$ has
at most $n+1$ connected components.
\end{teo}

\begin{proof}
First note that $x_{u}=-(A+uB)^{-1}[Cu+d]$ describes a smooth curve as long as
$\det(A+uB)\not =0$. Since $\det(A+uB)$ is a nontrivial polynomial in $u$ of
degree at most $n$, there are most $n$ real roots $v^{1},\ldots,v^{r},0\leq
r\leq n$, of $\det(A+uB)=0$. By our assumption the vectors $Cv^{i}+d$ are not
in the range of $A+v^{i}B$.

Consider a sequence $u^{k}\rightarrow v^{i}$ for some $i$. If $x_{u^{k}}$
remains bounded, we may assume that it converges to some $y\in\mathbb{R}^{n}$.
For $k\rightarrow\infty$ we find%
\[
(A+v^{i}B)y=-(Cv^{i}+d)
\]
contradicting the assumption of the theorem. It follows that $x_{u^{k}}$
becomes unbounded for $k\rightarrow\infty$.

(ii) If $\Omega=[u_{\ast},u^{\ast}],u_{\ast}<u^{\ast}$, the equilibrium set
$E=\{x_{u}\left\vert u\in\Omega\setminus\{v^{1},\ldots,v^{r}\}\right.  \}$
consists of at most $n+1$ smooth curves having no finite endpoints, with the
possible exception of the equilibria corresponding to the minimum and maximum
values of $u$ in $\Omega$, i.e., $u=u_{\ast},u^{\ast}$. If there is more than
one curve constituting $E$, then the finite end points which are the
equilibria $x_{u_{\ast}}$ and $x_{u^{\ast}}$ must lie on different curves.
Hence the assertion also follows in this case. Similarly, the assertion
follows for intervals which are unbounded to one side.
\end{proof}

The following example is used in Rink and Mohler \cite[Example 2]{RinkM68} and
Mohler \cite[Example 2 on page 32]{Mohler} as an example for a system that is
not controllable. It illustrates the result above.

\begin{example}
\label{Example2_RinkMohler}Consider the control system given by%
\[
\left[
\begin{array}
[c]{c}%
\dot{x}\\
\dot{y}%
\end{array}
\right]  =\left[
\begin{array}
[c]{c}%
2u(t)x+y\\
x+2u(t)y+u(t)
\end{array}
\right]  .
\]
With
\[
A=\left[
\begin{array}
[c]{cc}%
0 & 1\\
1 & 0
\end{array}
\right]  ,\quad B=\left[
\begin{array}
[c]{cc}%
2 & 0\\
0 & 2
\end{array}
\right]  ,\quad C=\left[
\begin{array}
[c]{c}%
0\\
1
\end{array}
\right]  ,
\]
this is the inhomogeneous bilinear control system%
\[
\left[
\begin{array}
[c]{c}%
\dot{x}\\
\dot{y}%
\end{array}
\right]  =\left[
\begin{array}
[c]{cc}%
2u & 1\\
1 & 2u
\end{array}
\right]  \left[
\begin{array}
[c]{c}%
x\\
y
\end{array}
\right]  +\left[
\begin{array}
[c]{c}%
0\\
1
\end{array}
\right]  u=\left(  A+uB\right)  \left[
\begin{array}
[c]{c}%
x\\
y
\end{array}
\right]  +Cu.
\]
The eigenvalues of $A(u)=A+uB$ are given by $0=\det\left(  A+uB\right)
=4u^{2}-1$, hence $\lambda_{1}(u)=2u+1>\lambda_{2}(u)=2u-1$. One finds
$\lambda_{2}(u)>0$ for $u>\frac{1}{2}$ and $\lambda_{1}(u)<0$ for $u<-\frac
{1}{2}$. For $u\in\left(  -\frac{1}{2},\frac{1}{2}\right)  $ one gets
$\lambda_{1}(u)>0$ and $\lambda_{2}(u)<0$, hence the matrix $A+uB$ is
hyperbolic here.

For every $u\in\mathbb{R}$, the eigenspace for $\lambda_{1}(u)$ is
$\mathrm{Diag}_{1}:=\{(z,z)^{\top}\left\vert z\in\mathbb{R}\right.  \}$ and
the eigenspace for $\lambda_{2}(u)$ is $\mathrm{Diag}_{2}:=\{(z,-z)^{\top
}\left\vert z\in\mathbb{R}\right.  \}$. For $\left\vert u\right\vert
\not =\frac{1}{2}$ the equilibria are given by%
\begin{equation}
\left[
\begin{array}
[c]{c}%
x_{u}\\
y_{u}%
\end{array}
\right]  =-\left(  A+uB\right)  ^{-1}Cu=\frac{-1}{4u^{2}-1}\left[
\begin{array}
[c]{cc}%
2u & -1\\
-1 & 2u
\end{array}
\right]  \left[
\begin{array}
[c]{c}%
0\\
1
\end{array}
\right]  u=\frac{u}{4u^{2}-1}\left[
\begin{array}
[c]{c}%
1\\
-2u
\end{array}
\right]  . \label{x_u}%
\end{equation}
Thus we see that%
\begin{equation}
y_{u}=-2ux_{u}\text{ for }\left\vert u\right\vert \not =\frac{1}{2}.
\label{components}%
\end{equation}
The assumption of Theorem \ref{Theorem2} is satisfied, since for $u=\pm
\frac{1}{2}$ there is no solution to
\[
\left[
\begin{array}
[c]{c}%
0\\
0
\end{array}
\right]  =(A+uB)\left[
\begin{array}
[c]{c}%
x\\
y
\end{array}
\right]  +Cu=\left[
\begin{array}
[c]{cc}%
\pm1 & 1\\
1 & \pm1
\end{array}
\right]  \left[
\begin{array}
[c]{c}%
x\\
y
\end{array}
\right]  +\left[
\begin{array}
[c]{c}%
0\\
1
\end{array}
\right]  \left(  \pm\frac{1}{2}\right)  .
\]
For the asymptotics of the equilibria, equation (\ref{components}) shows that
$(x_{u},y_{u})^{\top}$ approach the line $\mathrm{Diag}_{2}$ for
$u\rightarrow\frac{1}{2}$ and the line $\mathrm{Diag}_{1}$ for $u\rightarrow
-\frac{1}{2}$. In both cases, the equilibria become unbounded. For
$u\rightarrow\pm\infty$, one obtains that the equilibria approach
$(0,-\frac{1}{2})^{\top}$.

This discussion shows that the set of equilibria for unbounded control $u$
consists of the following three connected branches%
\begin{align*}
\mathcal{B}_{1}  &  =\left\{  \left[
\begin{array}
[c]{c}%
x_{u}\\
y_{u}%
\end{array}
\right]  \left\vert u\in\left(  -\frac{1}{2},\frac{1}{2}\right)  \right.
\right\}  ,\quad\mathcal{B}_{2}=\left\{  \left[
\begin{array}
[c]{c}%
x_{u}\\
y_{u}%
\end{array}
\right]  \left\vert u\in\left(  -\infty,-\frac{1}{2}\right)  \right.
\right\}  ,\\
\mathcal{B}_{3}  &  =\left\{  \left[
\begin{array}
[c]{c}%
x_{u}\\
y_{u}%
\end{array}
\right]  \left\vert u\in\left(  \frac{1}{2},\infty\right)  \right.  \right\}
.
\end{align*}
The equilibria in $\mathcal{B}_{2}$ and $\mathcal{B}_{3}$ both approach
$(0,-\frac{1}{2})^{\top}$ for $\left\vert u\right\vert \rightarrow\infty$; cf.
also Mohler \cite[Figure 2.1 on p. 33]{Mohler} or Rink and Mohler \cite[Figure
1]{RinkM68}. The equilibria in $\mathcal{B}_{2}$ are stable, those in
$\mathcal{B}_{3}$ are totally unstable, and those in $\mathcal{B}_{1}$ yield
one positive and one negative eigenvalue.
\end{example}

\section{Control sets and equilibria of affine systems\label{Section5}}

The controllability properties near equilibria will be analyzed assuming that
the linearized control systems are controllable. This yields results on the
control sets around equilibria.

In order to describe the properties of the system linearized about an
equilibrium, we recall the following classical result from Lee and\ Markus
\cite[Theorem 1 on p. 366]{LeeM67}.

\begin{theorem}
\label{Theorem_LeeMarkus67}Consider the control process in $\mathbb{R}^{n}$%
\begin{equation}
\dot{x}=f(x,u), \label{f}%
\end{equation}
where $f$ is $C^{1}$ and suppose that $f(0,0)=0$ where $0$ is in the interior
of the control range $\Omega$. Then the controllable set $\mathcal{O}^{-}(0)$
is open if, with $A=\frac{\partial f}{\partial x}(0,0)$ and $B=\frac{\partial
f}{\partial u}(0,0)$,%
\begin{equation}
\mathrm{rank}[B,AB,\ldots,A^{n-1}B]=n. \label{Kalman0}%
\end{equation}

\end{theorem}

Condition (\ref{Kalman0}) is the familiar Kalman condition for controllability
of the linearized system $\dot{x}=\frac{\partial f}{\partial x}(0,0)x+\frac
{\partial f}{\partial u}(0,0)u$ (without control restriction), cf. Sontag
\cite[Theorem 3, p. 89]{Son98}.

We apply this result to affine control systems and obtain that the reachable
set and the controllable set for an equilibrium are open, if the linearized
system is controllable.

\begin{proposition}
\label{Proposition_open}Consider the affine system (\ref{affine}) and let
$x_{u}$ be an equilibrium for a control value $u\in\mathrm{int}\left(
\Omega\right)  $, where the rank condition
\begin{equation}
\mathrm{rank}[B^{\prime}(u),A(u)B^{\prime}(u),\ldots,\left(  A(u)\right)
^{n-1}B^{\prime}(u)]=n \label{Kalman2}%
\end{equation}
holds with $B^{\prime}(u)$ defined by
\begin{equation}
B^{\prime}(u)=C+\left[  B_{1}x_{u},\ldots,B_{m}x_{u}\right]  . \label{B_prime}%
\end{equation}
Then the reachable set $\mathcal{O}^{+}(x_{u})$ and the controllable set
$\mathcal{O}^{-}(x_{u})$ are open. If $A(u)=A+\sum_{i=1}^{m}u_{i}B_{i}$ is
invertible, then%
\[
B^{\prime}(u)=C-\left[  B_{1}A(u)^{-1}(Cu+d),\ldots,B_{m}A(u)^{-1}%
(Cu+d)\right]  .
\]

\end{proposition}

\begin{proof}
First we convince ourselves that Theorem \ref{Theorem_LeeMarkus67} can be
applied to arbitrary equilibria $(x^{0},u^{0})$ with $u^{0}\in\mathrm{int}%
\left(  \Omega\right)  $ instead of $(0,0)$. In fact, define $\tilde
{f}(x,u):=f(x+x^{0},u+u^{0})$. Then $(0,0)$ is an equilibrium of%
\begin{equation}
\dot{x}(t)=\tilde{f}(x(t),u(t))\text{ with control range }\Omega-u^{0},
\label{f_tilde}%
\end{equation}
and the control value $u=0$ is in $\mathrm{int}(\Omega-u^{0})$. The solutions
$\psi(t,0,u),t\geq0$, of (\ref{f_tilde}) are given by $\varphi(t,x^{0}%
,u+u^{0})-x^{0}$, since $\varphi(0,x^{0},u+u^{0})-x^{0}=0$ and
\[
\frac{d}{dt}\left[  \varphi(t,x^{0},u+u^{0})-x^{0}\right]  =f(\varphi
(t,x^{0},u+u^{0}),u(t)+u^{0})=\tilde{f}(\varphi(t,x^{0},u+u^{0})-x^{0},u(t)).
\]
Hence $\mathcal{O}^{-}(x^{0})$ coincides with the controllable set
$\mathcal{\tilde{O}}^{-}(0)$ of (\ref{f_tilde}). The rank condition
(\ref{Kalman0}) for (\ref{f_tilde}) involves%
\[
A=\frac{\partial\tilde{f}}{\partial x}(0,0)=\frac{\partial f}{\partial
x}(x^{0},u^{0}),B=\frac{\partial\tilde{f}}{\partial u}(0,0)=\frac{\partial
f}{\partial u}(x^{0},u^{0}).
\]
For system (\ref{affine}) $f(x,u)=A(u)x+Cu+d$ and for an equilibrium $x_{u}$
we find $\frac{\partial f}{\partial x}(x_{u},u)=A(u)$ and%
\[
\frac{\partial f}{\partial u}(x_{u},u)=C+\frac{\partial}{\partial u}\sum
_{i=1}^{m}u_{i}B_{i}x_{u}=C+\left[  B_{1}x_{u},\ldots,B_{m}x_{u}\right]  .
\]
By (\ref{Kalman2}) the rank condition (\ref{Kalman0}) is satisfied. Applying
Theorem \ref{Theorem_LeeMarkus67} we conclude that the controllable set
$\mathcal{O}^{-}(x_{u})$ is open. By time reversal, cf. Lemma
\ref{Lemma_time_reversal}, also the reachable set $\mathcal{O}^{+}(x_{u})$ is open.

If $A(u)$ is invertible, the formula for $B^{\prime}(u)$ follows from
(\ref{equilibrium2}).
\end{proof}

The following proposition shows that the controllability rank condition
(\ref{Kalman2}) holds generically for controls $u\in\mathbb{R}^{m}$ if it
holds in some $u^{0}$.

\begin{proposition}
\label{Proposition_generic}Assume that $A(u)$ is invertible for all
$u\in\mathbb{R}^{m}$ and that the rank condition (\ref{Kalman2}) holds for
some $u^{0}\in\mathbb{R}^{m}$. Then (\ref{Kalman2}) holds for all $u$ in an
open and dense subset of $\mathbb{R}^{m}$.
\end{proposition}

\begin{proof}
Define
\[
B^{\prime\prime}(u):=\det A(u)C-\left[  B_{1}\mathrm{Adj}(A(u))(Cu+d),\ldots
,B_{m}\mathrm{Adj}(A(u))(Cu+d)\right]  ,
\]
where $\mathrm{Adj}(A(u))$ is defined by $\left(  A(u)\right)  ^{-1}\det
A(u)=\mathrm{Adj}(A(u))$. Condition (\ref{Kalman2}) holds if and only if%
\begin{equation}
\mathrm{rank}[B^{\prime\prime}(u),A(u)B^{\prime\prime}(u),\ldots,\left(
A(u)\right)  ^{n-1}B^{\prime\prime}(u)]=n. \label{Kalman5}%
\end{equation}
The entries of the matrix in (\ref{Kalman5}) are polynomial in the variables
$u_{1},\ldots,u_{m}$. Using the assumption one finds that the set of
$u\in\mathbb{R}^{m}$ violating (\ref{Kalman5}) is contained in a proper
algebraic variety; the complement of such a set is open and dense in
$\mathbb{R}^{m}$ (this follows in the same way as the genericity of the
controllability rank condition (\ref{Kalman0}), cf. Sontag \cite[Proposition
3.3.12]{Son98}).
\end{proof}

\begin{remark}
\label{Remark5.4}For a system of the form (\ref{affine}) with scalar control
the assumptions of Proposition \ref{Proposition_generic} imply that there are
at most finitely many\ $u$ such that the rank condition (\ref{Kalman2}) is
violated. This follows taking into account that for scalar $u$ the entries of
the matrix in (\ref{Kalman5}) are polynomial in the scalar variable $u$, hence
there are at most finitely many zeros.
\end{remark}

A consequence of Proposition \ref{Proposition_open} is the following first
result on control sets.

\begin{proposition}
\label{Proposition3}Consider the affine system (\ref{affine}) and assume that
the rank condition (\ref{Kalman2}) is satisfied for some $u\in\mathrm{int}%
(\Omega)$. Then the set $D=\mathcal{O}^{-}{(x_{u})}\cap\overline
{\mathcal{O}^{+}{(x_{u})}}$ is a control set of system (\ref{affine})
containing the equilibrium $x_{u}$ in the interior.
\end{proposition}

\begin{proof}
By Proposition \ref{Proposition_open} the sets $\mathcal{O}^{-}{(x_{u})}$ and
$\mathcal{O}^{+}{(x_{u})}$ are open neighborhoods of $x_{u}$, hence it follows
that $x_{u}$ is in the interior of the set $D_{0}:=\mathcal{O}^{-}{(x_{u}%
)}\cap\overline{\mathcal{O}^{+}{(x_{u})}}$.

Let $x\in D_{0}$. Then $x_{u}\in\mathcal{O}^{+}{(x)}$ and therefore
$\mathcal{O}^{+}{(x_{u})}\subset\mathcal{O}^{+}{(x)}$ and as $D_{0}%
\subset\overline{\mathcal{O}^{+}{(x_{u})}},$ it follows that $D_{0}%
\subset\overline{\mathcal{O}^{+}{(x)}}$. Next we show that there is a control
$v\in\mathcal{U}$ with $\varphi(t,x,v)\in D_{0}$ for all $t\geq0$. Since
$x\in\mathcal{O}^{-}(x_{u})$ there are $T>0$ and $v_{1}\in\mathcal{U}$ such
that $\varphi(T,x,v_{1})=x_{u}$ and $\varphi(t,x,v_{1})\in\mathcal{O}%
^{-}(x_{u})$ for all $t\in\lbrack0,T]$. Furthermore, $\varphi(t,x,v_{1}%
)\in\mathcal{O}^{+}(x)$ and $x\in\overline{\mathcal{O}^{+}(x_{u})},$ and hence
continuous dependence on the initial value shows that $\varphi(t,x,v_{1}%
)\in\overline{\mathcal{O}^{+}(x_{u})}$ for all $t\in\lbrack0,T].$ Now the
control function%
\[
v(t):=\left\{
\begin{array}
[c]{lll}%
v_{1}(t) & \text{for} & t\in\lbrack0,T]\\
u & \text{for} & t>T
\end{array}
\right.
\]
yields $\varphi(t,x,v)\in D_{0}$ for all $t\geq0.$ We have shown that $D_{0}$
satisfies properties (i) and (ii) in Definition \ref{def:Dset}. Hence it is
contained in a maximal set $D$ with these properties, i.e., a control set,
obtained as the union of all sets satisfying properties (i) and (ii) and
containing $D_{0}$.

Let us show that $D_{0}=D$.\ By the definition of control sets and $x_{u}\in
D$, the inclusion $D\subset\overline{\mathcal{O}^{+}{(x}_{u}{)}}$ holds and
for $x\in D$ one has $x_{u}\in\overline{\mathcal{O}^{+}(x)}$. Using that
$\mathcal{O}^{-}(x_{u})$ is a neighborhood of $x_{u}$ this implies that there
are $T>0$ and a control $u\in\mathcal{U}$ with $\varphi(T,x,u)\in
\mathcal{O}^{-}(x_{u})$, and hence $x\in\mathcal{O}^{-}(x_{u})$. This shows
that $D\subset\mathcal{O}^{-}(x_{u})\cap\overline{\mathcal{O}^{+}{(x}_{u}{)}%
}=D_{0}$ and hence equality holds concluding the proof that $D_{0}$ is a
control set.
\end{proof}

Next we show that every connected subset of the set $E_{0}$ of equilibria is
contained in a single control set, if the systems linearized about the
equilibria are controllable.

\begin{theorem}
\label{Theorem3}Let $\mathcal{C}\subset\{x_{u}\left\vert u\in\mathrm{int}%
\left(  \Omega\right)  \right.  \}=E_{0}$ be a pathwise connected subset of
the set of equilibria of system (\ref{affine}) and assume that for every
equilibrium $x_{u}$ in $\mathcal{C}$ the control $u$ satisfies the rank
condition (\ref{Kalman2}). Then there exists a control set $D$ containing
$\mathcal{C}$ in the interior and $D=\mathcal{O}^{-}{(x_{u})}\cap
\overline{\mathcal{O}^{+}{(x_{u})}}$ for every $x_{u}\in\mathcal{C}$.
\end{theorem}

\begin{proof}
By Proposition \ref{Proposition3} every equilibrium $x_{u}\in\mathcal{C}$ is
contained in the interior of a control set. Consider two points $x_{u}$ and
$x_{v}$ in $\mathcal{C}$. Then $x_{v}\in\mathcal{O}^{+}(x_{u})$. In fact,
consider a continuous path from $x_{u}$ to $x_{v}$ in $\mathcal{C}$, say
$h:[0,1]\rightarrow\mathcal{C}$ with $h(0)=x_{u}$ and $h(1)=x_{v}$. Let
\[
\tau:=\sup\{s\in\lbrack0,1]\left\vert \forall s^{\prime}\in\lbrack
0,s]:h(s^{\prime})\in\mathcal{O}^{+}(x_{u})\right.  \}.
\]
Observe that $\tau>0$, since by Proposition \ref{Proposition_open}, the
reachable set $\mathcal{O}^{+}(x_{u})$ is open. If $\tau<1$, then
$y:=h(\tau)\in\overline{\mathcal{O}^{+}(x_{u})}\setminus\mathcal{O}^{+}%
(x_{u})\subset\partial\mathcal{O}^{+}(x_{u})$. Thus $\mathcal{O}^{-}%
(y)\cap\mathcal{O}^{+}(x_{u})=\varnothing$. On the other hand, $y$ is an
equilibrium corresponding to a control in the interior of $\Omega$. Again
Proposition \ref{Proposition_open} implies that $\mathcal{O}^{-}(y)$ is a
neighborhood of $y$, and hence $\mathcal{O}^{-}(y)\cap\mathcal{O}^{+}%
(x_{u})\not =\varnothing$. This contradiction shows that $\tau=1$ and
$y=x_{v}$. Thus one can steer the system from any point $x_{u}\in\mathcal{C}$
to any other point $x_{v}\in\mathcal{C}$. It follows that $\mathcal{C}$ is
contained in a single control set $D$. The same arguments show that, in fact,
$\mathcal{C}$ is contained in the interior of $D$.
\end{proof}

\begin{remark}
For scalar control, Theorem \ref{Theorem2} shows that there are at most $n+1$
connected components of the set $E$ of equilibria, which consists of at most
$n+1$ smooth curves. Thus also $E_{0}$ consists of at most $n+1$ smooth curves
which, naturally, are pathwise connected. Hence, under the assumptions of
Theorem \ref{Theorem3}, there are at most $n+1$ control sets containing an
equilibrium in the interior.
\end{remark}

In the rest of this section, we relate the controllability properties of
system (\ref{affine}) to spectral properties of the matrices $A(u),u\in\Omega$.

\begin{lem}
\label{l1}Consider the affine system (\ref{affine}) and suppose that $x_{u}$
is an equilibrium for a control value $u\in\mathrm{int}(\Omega)$ satisfying
the rank condition (\ref{Kalman2}).

(i) If every eigenvalue of $A(u)$ has negative real part, it follows that
$\mathcal{O}^{-}(x_{u})=\mathbb{R}^{n}$.

(ii) If every eigenvalue of $A(u)$ has positive real part, it follows that
$\mathcal{O}^{+}(x_{u})=\mathbb{R}^{n}$.
\end{lem}

\begin{proof}
By Proposition \ref{Proposition_open} the rank condition (\ref{Kalman2})
implies that $\mathcal{O}^{-}(x_{u})$ and $\mathcal{O}^{+}(x_{u})$ are open.

(i) Let $0<\alpha<-\max\{\operatorname{Re}\lambda\left\vert \lambda\text{ an
eigenvalue of }A(u)\right.  \}$. Then there is a constant $c_{0}\geq1$ such
that every solution of the autonomous linear differential equation $\dot
{x}(t)=A(u)x(t),\,x(0)=x_{0}$, satisfies%
\begin{equation}
\left\Vert e^{A(u)t}x_{0}\right\Vert \leq c_{0}e^{-\alpha t}\left\Vert
x_{0}\right\Vert \text{ for all }t\geq0\text{.} \label{stab1}%
\end{equation}
The variation-of-constants formula applied for $x\in\mathbb{R}^{n}$ and
$x_{u}$ shows that%
\begin{align*}
&  \varphi(t,x,u)-x_{u}\\
&  =e^{A(u)t}x+\int_{0}^{t}e^{A(u)(t-s)}[Cu+d]ds-e^{A(u)t}x_{u}-\int_{0}%
^{t}e^{A(u)(t-s)}[Cu+d]ds\\
&  =e^{A(u)t}\left(  x-x_{u}\right)  .
\end{align*}
Thus (\ref{stab1}) implies%
\[
\left\Vert \varphi(t,x,u)-x_{u}\right\Vert \leq c_{0}e^{-\alpha t}\left\Vert
x-x_{u}\right\Vert \rightarrow0\text{ for }t\rightarrow\infty.
\]
Since $\mathcal{O}^{-}(x_{u})$ is a neighborhood of $x_{u}$, there exists
$T>0$ such that $\varphi(T,x,u)\in\mathcal{O}^{-}(x_{u})$. Thus $x\in
\mathcal{O}^{-}\left(  \varphi(T,x,u)\right)  \subset\mathcal{O}^{-}(x_{u})$
and $\mathbb{R}^{n}=\mathcal{O}^{-}(x_{u})$ follows.

(ii) For the system $\dot{x}(t)=-A(u)x-Cu-d$, every eigenvalue of $-A(u)$ has
negative real part. By (i) and time reversal, Lemma \ref{Lemma_time_reversal},
the assertion follows.
\end{proof}

\begin{remark}
An easy consequence of this lemma is that the system is controllable if there
are $u,v\in\Omega$ with equilibria $x_{u},x_{v}$ in the same pathwise
connected subset of $E_{0}$ such that every eigenvalue of $A(u)$ has negative
real part and every eigenvalue of $A(v)$ has positive real part; cf. Mohler
\cite[Main Result, p. 28]{Mohler} for the special case of inhomogeneous
bilinear systems of the form (\ref{bilinear}).
\end{remark}

The following corollary to Theorem \ref{Theorem3} shows that there is a
control set around the set of equilibria for uniformly hyperbolic matrices
$A(u),u\in\Omega$.

\begin{corollary}
\label{Corollary_hyperbolic}Consider an affine control system of the form
(\ref{affine}) and assume that

(i) the control range $\Omega=\overline{\mathrm{int}(\Omega)}$ is compact and
$\mathrm{int}\left(  \Omega\right)  $ is pathwise connected;

(ii) the matrices $A(u)$ are uniformly hyperbolic in the following sense:
There is $k$ with $0\leq k\leq n$ such that for all $u\in\Omega$ there are $k$
eigenvalues with $\operatorname{Re}\lambda_{1}(u),\ldots,\allowbreak
\operatorname{Re}\lambda_{k}(u)<0$ and $n-k$ eigenvalues with
$\operatorname{Re}\lambda_{k+1}(u),\ldots,\operatorname{Re}\lambda_{n}(u)>0$;

(iii) every $u\in\mathrm{int}\left(  \Omega\right)  $ satisfies the rank
condition (\ref{Kalman2}).

Then the set $E=\overline{E_{0}}$ of equilibria is compact and connected, the
set $E_{0}$ is pathwise connected, and there exists a control set $D$ with
$E_{0}\subset\mathrm{int}(D)$.
\end{corollary}

\begin{proof}
First observe that all matrices $A(u),\,u\in\Omega$, are invertible, since $0$
is not an eigenvalue. Thus the set $E=\{x_{u}\left\vert u\in\Omega\right.  \}$
of equilibria is compact and $E_{0}$ is pathwise connected, since $x_{u}$
depends continuously on $u$. By Theorem \ref{Theorem3} there exists a control
set containing $E_{0}$ in the interior. Since pathwise connected sets are
connected the set $\mathrm{int}\left(  \Omega\right)  $ is connected, which
implies that also $\Omega=\overline{\mathrm{int}(\Omega)}$ is connected, cf.
Engelking \cite[Corollary 6.1.11]{Engel}. It also follows that the set
$E=\overline{E_{0}}$ is connected.
\end{proof}

If condition (ii) of Corollary \ref{Corollary_hyperbolic} holds with $k=0$ or
$k=n$, i.e., if all matrices $A(u)$ are stable or all are totally unstable,
the rank condition (iii) for the linearized systems can be weakened.

\begin{corollary}
\label{Corollary_stable}Let assumption (i) of Corollary
\ref{Corollary_hyperbolic} be satisfied and assume that there are at most
finitely many\ points in $\mathrm{int}\left(  \Omega\right)  $ such that the
rank condition (\ref{Kalman2}) is violated.

(i) If for all $u\in\mathrm{int}\left(  \Omega\right)  $ all eigenvalues of
$A(u)$ have negative real parts, there exists a closed control set $D$ with
$E_{0}\subset\mathrm{int}(D)$.

(ii) If for all $u\in\mathrm{int}\left(  \Omega\right)  $ all eigenvalues of
$A(u)$ have positive real parts, there exists a control set $D$ with
$E_{0}\subset\mathrm{int}(D)$.
\end{corollary}

\begin{proof}
As in Corollary \ref{Corollary_hyperbolic}(i) it follows that the set $E_{0}$
of equilibria is pathwise connected. Consider equilibria $x_{u},x_{v}\in
E_{0}$ with $u,v\in\mathrm{int}\left(  \Omega\right)  $ and suppose that
$x_{u}$ satisfies condition (\ref{Kalman2}). Hence there is a control set
$D_{u}$ containing $x_{u}$ in the interior. We use a construction similar to
the one in the proof of Theorem \ref{Theorem3}: There is a continuous map
$h:[0,1]\rightarrow E_{0}$ with $h(0)=x_{u}$ and $h(1)=x_{v}$. Let%
\[
\tau:=\sup\{s\in\lbrack0,1]\left\vert \forall s^{\prime}\in\lbrack
0,s]:h(s^{\prime})\in D_{u}\right.  \}.
\]
Observe that $\tau>0$, since $x_{u}\in\mathrm{int}\left(  D_{u}\right)  $. If
$\tau<1$, then $y:=h(\tau)\in\partial D_{u}$ and $y=x_{w}$ is an equilibrium
for some $w\in\mathrm{int}\left(  \Omega\right)  $. If $w$ satisfies
(\ref{Kalman2}), then\ by Proposition \ref{Proposition3} $x_{w}$ is in the
interior of a control set contradicting the choice of $\tau$. It remains to
discuss the case where $w$ violates (\ref{Kalman2}).

(i) Since all eigenvalues of $A(u)$ have negative real parts, Lemma
\ref{l1}(i) implies that $x_{w}\in\mathcal{O}^{-}(x_{u})=\mathbb{R}^{n}$.
Hence one can steer $x_{w}$ (in finite time) into the interior of $D_{u}$, and
by continuous dependence on the initial value, this holds for all $x$ in a
neighborhood $N(x_{w})$. Note that $x_{w}\in\overline{D_{u}}\cap\partial
D_{u}$. Since there are only finitely many points violating (\ref{Kalman2}),
all points $h(s^{\prime\prime})$ with $s^{\prime\prime}\in(\tau,\tau
+\varepsilon)$ for some $\varepsilon>0$ satisfy (\ref{Kalman2}) and hence they
are in a single control set $D^{\prime}$ and hence $x_{w}\in\overline
{D^{\prime}}$. Then all points in the nonvoid intersection $N(x_{w})\cap
D^{\prime}$ can be steered into $D_{u}$. The same arguments show that one can
steer points in $D_{u}$ into $D^{\prime}$, hence $D^{\prime}=D_{u}$. This
contradicts the choice of $\tau$. It follows that $\tau=1$ and $x_{v}%
\in\overline{D_{u}}$. Using $x_{v}\in\mathcal{O}^{-}(x_{u})=\mathbb{R}^{n}$
and $D_{u}=\overline{\mathcal{O}^{+}(x_{u})}\cap\mathcal{O}^{-}(x_{u}%
)=\overline{\mathcal{O}^{+}(x_{u})}$ one sees that $x_{v}\in D_{u}$. We
conclude that all equilibria in $E_{0}$ are contained in the interior of a
single closed control set.

(ii) Since all eigenvalues of $A(u)$ have positive real parts, Lemma
\ref{l1}(ii) implies that $x_{w}\in\mathcal{O}^{+}(x_{u})=\mathbb{R}^{n}$.
This shows that $x_{w}$ can be reached from $x_{u}\in\mathrm{int}\left(
D_{u}\right)  $. Continuous dependence on the initial value shows that all
points in a neighborhood $N(x_{w})$ of $x_{w}$ can be reached from the
interior of $D_{u}$. Since there are only finitely many points violating
(\ref{Kalman2}), all points $h(s^{\prime\prime})$ with $s^{\prime\prime}%
\in(\tau,\tau+\varepsilon)$ for some $\varepsilon>0$ are in a single control
set $D^{\prime}$ and $x_{w}\in\overline{D^{\prime}}$. Then all points in the
nonvoid intersection $N(x_{w})\cap D^{\prime}$ can be reached from the
interior of $D_{u}$. The same arguments show that some point in $\mathrm{int}%
\left(  D_{u}\right)  $ can be reached from $D^{\prime}$, hence $D^{\prime
}=D_{u}$. This contradicts the choice of $\tau$. It follows that $\tau=1$ and
$x_{v}\in\overline{D_{u}}$. Using $x_{v}\in\mathcal{O}^{+}(x_{u}%
)=\mathbb{R}^{n}$ and $D_{u}=\overline{\mathcal{O}^{+}(x_{u})}\cap
\mathcal{O}^{-}(x_{u})=\mathcal{O}^{-}(x_{u})$ one sees that $x_{v}\in D_{u}$.
We conclude that all equilibria in $E_{0}$ are contained in the interior of a
single control set.
\end{proof}

\begin{remark}
Remark \ref{Remark5.4} shows for an affine system of the form (\ref{affine})
with scalar control satisfying the assumptions of Proposition
\ref{Proposition_generic} that there are at most finitely many\ points $u$
where the rank condition (\ref{Kalman2}) is violated.
\end{remark}

Next we provide a sufficient condition for the existence of unbounded control sets.

\begin{teo}
\label{Theorem_unbounded}Consider an affine control system of the form
(\ref{affine}), let $\mathcal{C}$ be a pathwise connected subset of the set
$E_{0}$ of equilibria of system (\ref{affine}) and define $\Omega
(\mathcal{C})=\{u\in\mathrm{int}(\Omega)\left\vert x_{u}\in\mathcal{C}\right.
\}$. Assume that

(i) there is $u^{0}\in\overline{\Omega(\mathcal{C})}$ such that $A(u^{0})$ has
the eigenvalue $\lambda_{0}=0$ and $Cu^{0}+d$ is not in the range of
$A(u^{0})$;

(ii) every $u\in\Omega(\mathcal{C}),u\not =u^{0}$, satisfies $\mathrm{rank}%
A(u)=n$ and the rank condition (\ref{Kalman2}).

Then, there is an unbounded control set $D\subset\mathbb{R}^{n}$ containing
$\mathcal{C}$ in the interior. More precisely, for $u^{k}\in\Omega
(\mathcal{C})$ with $u^{k}\rightarrow u^{0}$ for $k\rightarrow\infty$, the
equilibria $x_{u^{k}}\in\mathcal{C}\subset\mathrm{int}(D)$ satisfy for
$k\rightarrow\infty$
\begin{equation}
\left\Vert x_{u^{k}}\right\Vert \rightarrow\infty\text{ and }\frac{x_{u^{k}}%
}{\left\Vert x_{u^{k}}\right\Vert }\rightarrow\ker A(u^{0})\cap\mathbb{S}%
^{n-1}. \label{5.8}%
\end{equation}

\end{teo}

\begin{proof}
By Theorem \ref{Theorem3} there is a control set $D$ containing $\mathcal{C}$
in the interior. In order to show that $D$ is unbounded, we argue similarly as
in the scalar situation in Theorem \ref{Theorem2}.

Let $u^{k}\in\Omega(\mathcal{C})$ converge to $u^{0}$ and assume, by way of
contradiction, that $x_{u^{k}}$ remains bounded, hence we may suppose that
there is $x^{0}\in\mathbb{R}^{n}$ with $x_{u^{k}}\rightarrow x^{0}$. Then the
equalities%
\[
A(u^{k})x_{u^{k}}=-\left[  Cu^{k}+d\right]
\]
lead for $k\rightarrow\infty$ to%
\[
A(u^{0})x_{u^{0}}=-\left[  Cu^{0}+d\right]
\]
contradicting assumption (i). We have shown that $x_{u^{k}}$ becomes unbounded
for $k\rightarrow\infty$. Since $Cu^{k}+d\rightarrow Cu^{0}+d$, we get%
\[
A(u^{k})\frac{x_{u^{k}}}{\left\Vert x_{u^{k}}\right\Vert }=\frac{1}{\left\Vert
x_{u^{k}}\right\Vert }\left(  Cu^{k}+d\right)  \rightarrow0.
\]
On the other hand, every cluster point $y\in\mathbb{R}^{n}$ of the bounded
sequence $\frac{x_{u^{k}}}{\left\Vert x_{u^{k}}\right\Vert }$ satisfies
$\left\Vert y\right\Vert =1$ and (\ref{5.8}) follows.
\end{proof}

Theorem \ref{Theorem_unbounded} sheds some light on the relation between
controllability properties of affine systems and their homogeneous bilinear
parts: By Theorem \ref{Theorem_95} assumption (i) is related to the existence
of a control set of the latter system in $\mathbb{R}^{n}$\textbf{.}

We state the following result concerning closed invariant cones (cf. Remark
\ref{Remark_inv_cone}). This is formulated in the context of semigroup
actions. Denote by $S_{\mathrm{aff}}$ and $S_{\mathrm{\hom}}$ the system
semigroups of the affine and the homogeneous bilinear control systems given by
(\ref{affine}) and (\ref{bilinear_h}), respectively. They correspond to
piecewise constant controls (see Appendix A of \cite{ColK00}). The system
group of the affine control system is given by the semidirect product
$G=H\rtimes\mathbb{R}^{n}$, where $H$ is the system group of the homogenous
bilinear system. The affine group operation is defined by $(g,v)\cdot
(h,w)=(gh,v+gw)$ for all $(g,v),(h,w)\in G$, and the affine action of $G$ on
$\mathbb{R}^{n}$ is given by $(g,v)\cdot w=gw+v$ with $(g,v)\in G$ and
$w\in\mathbb{R}^{n}$ using the linear action of $H$ on $\mathbb{R}^{n}$. A set
$Q\subset\mathbb{R}^{n}$ is invariant under $S_{\mathrm{aff}}$ and
$S_{\mathrm{\hom}}$ if and only if it is invariant for the affine control
system and the homogeneous bilinear control systems, respectively. We get the
following relations between invariance of a closed cone for $S_{\mathrm{aff}}$
and $S_{\mathrm{\hom}}$.

\begin{proposition}
\label{Proposition5.14}Consider an affine control system of the form
(\ref{affine}) and its homogeneous bilinear part (\ref{bilinear_h}), and let
$K$ be a closed cone in $\mathbb{R}^{n}$.

(i) Suppose that $K$ is invariant for the homogeneous bilinear part and
$Cu+d\in K$ for all $u\in\Omega$. Then $K$ is invariant for the affine control system.

(ii) If $K$ is invariant for the affine control system, then it is invariant
for the homogeneous bilinear part.
\end{proposition}

\begin{proof}
Assertion (i) is immediate from the definitions. The assumption in (ii) means
that $(g,v)\cdot w\in K$ for all $(g,v)\in S_{\mathrm{aff}}$ and $w\in K$.
Suppose, by way of contradiction, that there exists $g\in S_{\mathrm{\hom}}$
with $x:=gw\notin K$ for some $w\in K$; hence $g(\lambda w)=\lambda
(gw)=\lambda x\not \in K$ for all $\lambda>0$. It follows that%
\[
\inf\{\left\Vert \lambda x-\lambda w^{\prime}\right\Vert \left\vert w^{\prime
}\in K\right.  \}=\lambda\inf\{\left\Vert x-w^{\prime}\right\Vert \left\vert
w^{\prime}\in K\right.  \}\rightarrow\infty\text{ for }\lambda\rightarrow
\infty.
\]
Hence for every $v\in\mathbb{R}^{n}$ there is $\lambda>0$ such that
$\inf\{\left\Vert g(\lambda w)+v-w^{\prime}\right\Vert \left\vert w^{\prime
}\in K\right.  \}>0$ implying $g(\lambda w)+v\not \in K$. This means for the
action of $S_{\mathrm{aff}}$ that $(g,v)\cdot(\lambda w)=g(\lambda
w)+v\not \in K$ contradicting the invariance of $K$ for $S_{\mathrm{aff}}$.
\end{proof}

\begin{remark}
Jurdjevic and Sallet \cite[Theorem 2]{JurS84} shows that controllability of an
affine control system without fixed points can be guaranteed if its
homogeneous bilinear part is controllable. Furthermore, for $Q\subset
\mathbb{R}^{n}$ let $\mathrm{A}(Q)$ be its affine hull. Suppose that $Q$ is
invariant for the affine control system. Then \cite[Lemma 3]{JurS84} implies
that $\mathrm{A}(Q)$ is invariant for the affine control system and the set
$\{\sum_{i=1}^{p}\left.  \lambda_{i}q_{i}\right\vert q_{i}\in Q,{\lambda_{i}%
}\in\mathbb{R}$ with $\sum_{i=1}^{p}{\lambda}_{i}=0,p\in\mathbb{N}\}$ is
invariant for its homogeneous bilinear part.
\end{remark}

Finally, we illustrate Theorem \ref{Theorem3} and Theorem
\ref{Theorem_unbounded} by discussing the control sets for two affine systems.
Recall that by Theorem \ref{Theorem_95}, the existence of a control $u^{0}%
\in\mathrm{int}(\Omega)$ such that $0$ is an eigenvalue of $A(u^{0})$ is
connected with the existence of an unbounded control set of the bilinear
system $\dot{x}=A(u)x$.

\begin{example}
\label{Example2_RinkMohler_2}Consider again Example \ref{Example2_RinkMohler}.
In order to describe the control sets we first check the controllability rank
condition (\ref{Kalman2}) for $\left\vert u\right\vert \not =\frac{1}{2}$. By
(\ref{x_u})%
\[
B^{\prime}(u)=C+Bx_{u}=\left[
\begin{array}
[c]{c}%
0\\
1
\end{array}
\right]  +\frac{u}{4u^{2}-1}\left[
\begin{array}
[c]{cc}%
2 & 0\\
0 & 2
\end{array}
\right]  \left[
\begin{array}
[c]{c}%
1\\
-2u
\end{array}
\right]  =\frac{1}{4u^{2}-1}\left[
\begin{array}
[c]{c}%
2u\\
-1
\end{array}
\right]  ,
\]
and hence%
\[
(4u^{2}-1)\left[  B^{\prime}(u),A(u)B^{\prime}(u)\right]  =\left[
\begin{array}
[c]{c}%
2u\\
-1
\end{array}
,\left(
\begin{array}
[c]{cc}%
2u & 1\\
1 & 2u
\end{array}
\right)  \left(
\begin{array}
[c]{c}%
2u\\
-1
\end{array}
\right)  \right]  =\left[
\begin{array}
[c]{cc}%
2u & 4u^{2}-1\\
-1 & 0
\end{array}
\right]  .
\]
Thus the rank condition (\ref{Kalman2}) holds in every equilibrium
$(x_{u},y_{u})$ with $\left\vert u\right\vert \not =\frac{1}{2}$.

Next we discuss the control sets for several control ranges given by a compact interval.

- Let $\Omega=[u_{\ast},u^{\ast}]$ with $\frac{1}{2}<u_{\ast}<u^{\ast}$. Then
the set of equilibria is given by the compact subset $\{(x_{u},y_{u}%
)\left\vert u\in\lbrack u_{\ast},u^{\ast}]\right.  \}\subset\mathcal{B}_{3}$.
By Theorem \ref{Theorem3} there is a single control set $D_{3}$ with
$(x_{u},y_{u})\in\mathrm{int}(D_{3})$ for all $u\in(u_{\ast},u^{\ast})$.

- Let $\Omega=[u_{\ast},u^{\ast}]$ with $u_{\ast}<u^{\ast}<-\frac{1}{2}$. Then
the set of equilibria is given by the compact subset $\{(x_{u},y_{u}%
)\left\vert u\in\lbrack u_{\ast},u^{\ast}]\right.  \}\subset\mathcal{B}_{2}$.
By Theorem \ref{Theorem3} there is a single closed control set $D_{2}$ with
$(x_{u},y_{u})\in\mathrm{int}(D_{2})$ for all $u\in(u_{\ast},u^{\ast})$.

- Let $\Omega=[u_{\ast},u^{\ast}]$ with $-\frac{1}{2}<u_{\ast}<u^{\ast}%
<\frac{1}{2}$. Then the set of equilibria is given by the compact subset
$\{(x_{u},y_{u})\left\vert u\in\lbrack u_{\ast},u^{\ast}]\right.
\}\subset\mathcal{B}_{1}$. By Theorem \ref{Theorem3} there is a single control
set $D_{1}$ with $(x_{u},y_{u})\in\mathrm{int}(D_{1})$ for all $u\in(u_{\ast
},u^{\ast})$.

- Let $\Omega=[-1,1]$. Then the connected components of the set $E_{0}$ of
equilibria are
\begin{align*}
\mathcal{C}_{1}  &  =\left\{  (x_{u},y_{u})\left\vert u\in\left(  -\frac{1}%
{2},\frac{1}{2}\right)  \right.  \right\}  ,\quad\mathcal{C}_{2}=\left\{
(x_{u},y_{u})\left\vert u\in\left(  -1,-\frac{1}{2}\right)  \right.  \right\}
,\\
\mathcal{C}_{3}  &  =\left\{  (x_{u},y_{u})\left\vert u\in\left(  \frac{1}%
{2},1\right)  \right.  \right\}  ,
\end{align*}
and there are control sets $D_{i}$ with $\mathcal{C}_{i}\subset\mathrm{int}%
\left(  D_{i}\right)  $ for $i=1,2,3$. Since these sets of equilibria are
unbounded also the control sets are unbounded. Based on Proposition
\ref{Proposition_shift}, a lengthy argument involving the phase portraits for
constant controls shows that one cannot steer the system from $D_{2}$ to
$D_{3}$ or $D_{1}$ and from $D_{1}$ to $D_{3}$, hence these control sets are
pairwise different.
\end{example}

Next we take up the linear oscillator from Example \ref{Example6} and consider
an associated affine control system. We will show that there are two unbounded
control sets.

\begin{example}
\label{Example6_2}Consider the affine control system given by%
\[
\ddot{x}+3\dot{x}+(1+u(t))x=u(t)+d\text{ with }u(t)\in\lbrack-\rho,\rho],
\]
where $\rho\in\left(  1,\frac{5}{4}\right)  $ and $d\in\mathbb{R}$. Hence the
system equation has the form%
\[
\left[
\begin{array}
[c]{c}%
\dot{x}\\
\dot{y}%
\end{array}
\right]  =\left[
\begin{array}
[c]{cc}%
0 & 1\\
-1 & -3
\end{array}
\right]  \left[
\begin{array}
[c]{c}%
x\\
y
\end{array}
\right]  +u(t)\left[
\begin{array}
[c]{cc}%
0 & 0\\
-1 & 0
\end{array}
\right]  \left[
\begin{array}
[c]{c}%
x\\
y
\end{array}
\right]  +u(t)\left[
\begin{array}
[c]{c}%
0\\
1
\end{array}
\right]  +\left[
\begin{array}
[c]{c}%
0\\
d
\end{array}
\right]  .
\]
For the equilibria with $u\not =-1$ we find%
\begin{equation}
\left[
\begin{array}
[c]{c}%
x_{u}\\
y_{u}%
\end{array}
\right]  =-\left[
\begin{array}
[c]{cc}%
0 & 1\\
-1-u & -3
\end{array}
\right]  ^{-1}\left[
\begin{array}
[c]{c}%
0\\
u+d
\end{array}
\right]  =\left[
\begin{array}
[c]{cc}%
\frac{3}{1+u} & \frac{1}{1+u}\\
-1 & 0
\end{array}
\right]  \left[
\begin{array}
[c]{c}%
0\\
u+d
\end{array}
\right]  =\left[
\begin{array}
[c]{c}%
\frac{d+u}{1+u}\\
0
\end{array}
\right]  . \label{equi}%
\end{equation}
This yields that the connected components of the set $E_{0}$ of equilibria are%
\[
\mathcal{C}_{1}=\left\{  \left.  \left[
\begin{array}
[c]{c}%
\frac{d+u}{1+u}\\
0
\end{array}
\right]  \right\vert u\in(-\rho,-1)\right\}  ,\quad\mathcal{C}_{2}=\left\{
\left.  \left[
\begin{array}
[c]{c}%
\frac{d+u}{1+u}\\
0
\end{array}
\right]  \right\vert u\in\left(  -1,\rho\right)  \right\}  .
\]
For $d=1$ there is a single equilibrium given by $(x_{u},y_{u})^{\top
}=(1,0)^{\top}$ for every $u\not =-1$. Henceforth we assume $d\not =1$.

Let $d<1$. Then for $u\in\lbrack-\rho,-1)$ one obtains $d+u<1+u<0$, and for
$u\in(-1,\rho]$ one obtains $1+u>0$, hence%
\[
\mathcal{C}_{1}=\left\{  \left.  \left[
\begin{array}
[c]{c}%
x\\
0
\end{array}
\right]  \right\vert x\in\left(  \frac{d-\rho}{1-\rho},\infty\right)
\right\}  ,\quad\mathcal{C}_{2}=\left\{  \left.  \left[
\begin{array}
[c]{c}%
x\\
0
\end{array}
\right]  \right\vert x\in\left(  -\infty,\frac{d+\rho}{1+\rho}\right)
\right\}  .
\]

Let $d>1$. Then $u\in\lbrack-\rho,-1)$ yields $1+u<0$ and $u\in(-1,\rho]$
yields $1+u>0$, hence%
\[
\mathcal{C}_{1}=\left\{  \left.  \left[
\begin{array}
[c]{c}%
x\\
0
\end{array}
\right]  \right\vert x\in\left(  -\infty,\frac{d-\rho}{1-\rho}\right)
\right\}  ,\quad\mathcal{C}_{2}=\left\{  \left.  \left[
\begin{array}
[c]{c}%
x\\
0
\end{array}
\right]  \right\vert x\in\left(  \frac{d+\rho}{1+\rho},\infty\right)
\right\}  .
\]
Note that $\mathcal{C}_{1}\cap\mathcal{C}_{2}=\varnothing$ for all $d$. The
equilibria in $\mathcal{C}_{1}$ are hyperbolic, since here $\lambda
_{1}(u)<0<\lambda_{2}(u)$ with $\lambda_{2}(u)\rightarrow0$ for $u\rightarrow
-1$. The equilibria in $\mathcal{C}_{2}$ are stable nodes since here
$\lambda_{1}(u)<\lambda_{2}(u)<0$.

Next we check the assumptions of Theorem \ref{Theorem_unbounded}. For
$u^{0}=-1$ the matrix $A(-1)=\left[
\begin{array}
[c]{cc}%
0 & 1\\
0 & -3
\end{array}
\right]  $ has the eigenvalue $\lambda_{0}=0$ with eigenspace $\mathbb{R}%
\times\{0\}$, and $\operatorname{Im}A(-1)=\{(y,-3y)\left\vert y\in
\mathbb{R}\right.  \}$. Furthermore
\[
Cu^{0}+d=\left[
\begin{array}
[c]{c}%
0\\
1
\end{array}
\right]  (-1)+\left[
\begin{array}
[c]{c}%
0\\
d
\end{array}
\right]  =\left[
\begin{array}
[c]{c}%
0\\
d-1
\end{array}
\right]
\]
is not in the range of $A(-1)$. This verifies assumption (i) in Theorem
\ref{Theorem_unbounded}. In order to check the rank condition (\ref{Kalman2})
we compute for $u\not =-1$%
\begin{align*}
B^{\prime}(u)  &  =C+B\left[
\begin{array}
[c]{c}%
x_{u}\\
y_{u}%
\end{array}
\right]  =\left[
\begin{array}
[c]{c}%
0\\
1
\end{array}
\right]  +\left[
\begin{array}
[c]{cc}%
0 & 0\\
-1 & 0
\end{array}
\right]  \left[
\begin{array}
[c]{c}%
\frac{d+u}{1+u}\\
0
\end{array}
\right]  =\left[
\begin{array}
[c]{c}%
0\\
\frac{1-d}{1+u}%
\end{array}
\right]  ,\\
A(u)B^{\prime}(u)  &  =\left[
\begin{array}
[c]{cc}%
0 & 1\\
-1-u & -3
\end{array}
\right]  \left[
\begin{array}
[c]{c}%
0\\
\frac{1-d}{1+u}%
\end{array}
\right]  =\left[
\begin{array}
[c]{c}%
\frac{1-d}{1+u}\\
-3\frac{1-d}{1+u}%
\end{array}
\right]  .
\end{align*}
Hence $\mathrm{rank}\left[  B^{\prime}(u),A(u)B^{\prime}(u)\right]  =2$ for
$u\not =-1$. Theorem \ref{Theorem_unbounded} implies that there are unbounded
control sets $D_{i}$ containing the equilibria in $\mathcal{C}_{i},i=1,2$, in
the interior. For $u^{k}\rightarrow u^{0}=-1$, the equilibria $(x_{u^{k}%
},y_{u^{k}})=(x_{u^{k}},0)$ become unbounded for $k\rightarrow\infty$ and%
\[
\frac{(x_{u^{k}},0)}{\left\Vert (x_{u^{k}},0)\right\Vert }\in\ker
A(-1)\cap\mathbb{S}^{1}=\left\{  \left[
\begin{array}
[c]{c}%
1\\
0
\end{array}
\right]  ,\left[
\begin{array}
[c]{c}%
-1\\
0
\end{array}
\right]  \right\}  \text{ for all }k.
\]
In the simple case considered here, the latter assertion is already clear by
formula (\ref{equi}) for the equilibria.

While the asymptotic stability of the equilibria in $\mathcal{C}_{2}$ implies
that one can steer the system from $\mathcal{C}_{1}$ to $\mathcal{C}_{2}$, the
converse does not hold which follows by inspection of the phase portraits for
the controls in $\left[  -\rho,-1\right]  $ and $\left[  -1,\rho\right]  $. It
follows that $D_{1}\not =D_{2}$.
\end{example}

\textbf{Acknowledgements.} We would like to thank two anonymous reviewers
whose comments helped to improve the paper.

\end{document}